\documentclass{amsart}
\usepackage{amsmath, amsfonts, amssymb,color}
\usepackage{stmaryrd,epsfig}
\usepackage{xypic}

\newtheorem{thm}{Theorem}[section] 
\newtheorem{lemma}[thm]{Lemma}
\newtheorem{prop}[thm]{Proposition}
\newtheorem{cor}[thm]{Corollary}
\newtheorem{conj}[thm]{Conjecture}
\newtheorem{cons}[thm]{Consequence}

\theoremstyle{definition}
\newtheorem{remark}[thm]{Remark}
\newtheorem{example}[thm]{Example}

\DeclareMathOperator{\Gr}{Gr}
\DeclareMathOperator{\Fl}{Fl}
\DeclareMathOperator{\LG}{LG}
\DeclareMathOperator{\IG}{IG}

\DeclareMathOperator{\OG}{OG}

\DeclareMathOperator{\Mat}{Mat}
\DeclareMathOperator{\Hom}{Hom}
\DeclareMathOperator{\Aut}{Aut}
\DeclareMathOperator{\GL}{GL}

\DeclareMathOperator{\Ker}{Ker}
\DeclareMathOperator{\Span}{Span}
\DeclareMathOperator{\Spec}{Spec}
\DeclareMathOperator{\QH}{QH}
\DeclareMathOperator{\QK}{QK}

\DeclareMathOperator{\codim}{codim}
\DeclareMathOperator{\im}{Im}
\newcommand{\bP}{{\mathbb P}}
\renewcommand{\P}{{\mathbb P}}
\newcommand{\N}{{\mathbb N}}
\newcommand{\Z}{{\mathbb Z}}
\newcommand{\Q}{{\mathbb Q}}
\newcommand{\C}{{\mathbb C}}
\newcommand{\cA}{{\mathcal A}}
\newcommand{\cB}{{\mathcal B}}
\newcommand{\cE}{{\mathcal E}}
\newcommand{\cF}{{\mathcal F}}

\newcommand{\cO}{{\mathcal O}}
\newcommand{\cS}{{\mathcal S}}

\newcommand{\cc}{\!:\!}

\newcommand{\euler}[1]{\chi_{_{#1}}}
\newcommand{\ch}{\operatorname{ch}}
\newcommand{\Td}{\operatorname{Td}}
\newcommand{\pt}{\text{point}}
\newcommand{\id}{\text{id}}
\newcommand{\bull}{{\scriptscriptstyle \bullet}}
\newcommand{\Bl}{\text{B$\ell$}}
\newcommand{\sh}{\operatorname{shape}}
\newcommand{\ev}{\operatorname{ev}}
\newcommand{\wt}{\widetilde}
\newcommand{\wh}{\widehat}
\newcommand{\wb}{\overline}
\newcommand{\hmm}[1]{\mbox{}\hspace{#1mm}\mbox{}}
\newcommand{\pic}[2]{\includegraphics[scale=#1]{#2}}

\newcommand{\ignore}[1]{}

\newcommand{\Mb}{\wb{\mathcal M}}
\newcommand{\pr}{\operatorname{pr}}
\newcommand{\groth}{{\mathcal G}}

\newenvironment{abcenum}{\begin{enumerate}}{\end{enumerate}}

\begin{document}

\title{Quantum $K$-theory of Grassmannians}

\date{April 18, 2009}

\author{Anders~S.~Buch}
\address{Department of Mathematics, Rutgers University, 110
  Frelinghuysen Road, Piscataway, NJ 08854, USA}
\email{asbuch@math.rutgers.edu}

\author{Leonardo~C.~Mihalcea} 
\address{Department of Mathematics, Duke University, P.O. Box 90320
  Durham, NC 27708-0320}
\email{lmihalce@math.duke.edu}

\subjclass[2000]{Primary 14N35; Secondary 19E08, 14M15, 14N15, 14E08}

\thanks{The first author was supported in part by NSF Grant
  DMS-0603822.}

\begin{abstract}
  We show that (equivariant) $K$-theoretic 3-point Gromov-Witten
  invariants of genus zero on a Grassmann variety are equal to triple
  intersections computed in the (equivariant) $K$-theory of a two-step
  flag manifold, thus generalizing an earlier result of Buch, Kresch,
  and Tamvakis.  In the process we show that the Gromov-Witten variety
  of curves passing through $3$ general points is irreducible and
  rational.  Our applications include Pieri and Giambelli formulas for
  the quantum $K$-theory ring of a Grassmannian, which determine the
  multiplication in this ring.  We also compute the dual Schubert
  basis for this ring, and show that its structure constants satisfy
  $S_3$-symmetry.  Our formula for Gromov-Witten invariants can be
  partially generalized to cominuscule homogeneous spaces by using a
  construction of Chaput, Manivel, and Perrin.
\end{abstract}

\maketitle

\section{Introduction}

The study of the (small) quantum cohomology ring began with Witten
\cite{witten:two-dimensional} and Kontsevich
\cite{kontsevich.manin:quantum} more than a decade ago, and has by now
evolved into a subject with deep ramifications in algebraic geometry,
representation theory and combinatorics, see e.g.\ 
\cite{bertram:quantum, fulton.pandharipande:notes,
  fomin.gelfand.ea:quantum, ciocan-fontanine:on, kim:quantum*1,
  rietsch:totally, postnikov:affine, lam.shimozono:quantum} and
references therein.

A result of Buch, Kresch, and Tamvakis
\cite{buch.kresch.ea:gromov-witten} reduced the computation of the
(3-point, genus 0) Gromov-Witten invariants of a Grassmann variety to
a computation in the ordinary cohomology of certain two-step flag
manifolds.  The Gromov-Witten invariants in question have an
enumerative interpretation: they count rational curves meeting general
translates of Schubert varieties.  The identity from
\cite{buch.kresch.ea:gromov-witten} was proved by establishing a
set-theoretic bijection between the curves counted by a Gromov-Witten
invariant and the points of intersection of three Schubert varieties
in general position in a two-step flag manifold.

The Gromov-Witten invariants used to define more general quantum
cohomology theories, such as equivariant quantum cohomology
\cite{givental:equivariant,kim:on*3} or quantum $K$-theory
\cite{givental:on,lee:quantum}, lack such an enumerative
interpretation. For example, the $K$-theoretic Gromov-Witten
invariants are equal to the sheaf Euler characteristic of
Gromov-Witten varieties of rational curves of fixed degree meeting
three general Schubert varieties.  Gromov-Witten varieties are
subvarieties of Kontsevich's moduli space of stable maps, and have
been studied by Lee and Pandharipande
\cite{pandharipande:canonical,lee.pandharipande:reconstruction}.

Despite lacking an enumerative interpretation, it turns out that the
more general Gromov-Witten invariants satisfy the same identity as the
one given earlier in \cite{buch.kresch.ea:gromov-witten}: an
equivariant $K$-theoretic Gromov-Witten invariant on a Grassmannian is
equal to a quantity computed in the ordinary equivariant $K$-theory of
a two-step flag variety.  This is our first main result.

There are two key elements in its proof.  The first is a commutative
diagram involving a variety $\Bl_d$, which dominates both Kontsevich's
moduli space of stable maps and a triple Grassmann-bundle over the
two-step flag variety from \cite{buch.kresch.ea:gromov-witten}.
Intuitively, for small degrees $d$, the variety $\Bl_d$ is the blow-up
of the moduli space along the locus of curves which have kernels and
spans (in the sense of \cite{buch:quantum}, see \S
\ref{ssection:qkgrass} below) of unexpected dimensions.  Our
commutative diagram combined with geometric properties of Kontsevich's
moduli space makes it possible to translate the computation of a
Gromov-Witten invariant from the moduli space to the Grassmann bundle,
where classical intersection theory provides the final answer.

This approach suffices to compute all the equivariant Gromov-Witten
invariants, and the (equivariant) $K$-theoretic invariants for small
degrees.  To compute $K$-theoretic invariants of large degrees $d$, we
need the second key element.  We will show that the Gromov-Witten
variety of curves of degree $d$ passing through $3$ general points is
irreducible and rational.  This gives a partial answer to a question
posed by Lee and Pandharipande in
\cite{lee.pandharipande:reconstruction} (see \S \ref{S:gwvar} below).

In addition to Grassmannians of type A, the formula for Gromov-Witten
invariants given in \cite{buch.kresch.ea:gromov-witten} was also
obtained for Lagrangian Grassmannians and maximal orthogonal
Grassmannians.  Chaput, Manivel, and Perrin later gave a
type-independent construction which generalized the formula to all
cominuscule spaces.  Our formula for Gromov-Witten invariants holds
partially in this generality.  To be precise, we will compute all the
(3 point, genus zero) equivariant Gromov-Witten invariants on any
cominuscule homogeneous space, as well as all the (equivariant)
$K$-theoretic invariants for small degrees.

Our main motivation for this paper was to determine the structure of
the {\em quantum $K$-theory ring\/} of a Grassmann variety.  This ring
was introduced by Givental and Lee, motivated by a study of the
relationship between Gromov-Witten theory and integrable systems
\cite{givental:on, givental.lee:quantum, lee:quantum}.  We describe
the ring structure in terms of a {\em Pieri rule\/} that shows how to
multiply arbitrary Schubert classes with special Schubert classes,
corresponding to the Chern classes of the universal quotient bundle.
Since the special Schubert classes generate the quantum $K$-theory
ring, this gives a complete combinatorial description of its
structure.  As an application, we prove a Giambelli formula that
expresses any Schubert class as a polynomial in the special classes;
this formula greatly simplifies the computation of products of
arbitrary Schubert classes.  We also prove that the structure
constants of the quantum $K$-theory ring satisfy $S_3$-symmetry in the
sense that they are invariant under permutations of indexing
partitions, and that the basis of Schubert structure sheaves can be
dualized by multiplying all classes with a constant element.  We
remark that unlike ordinary quantum cohomology, the structure
constants in quantum $K$-theory are not single Gromov-Witten
invariants (see \cite{givental:on} and \S \ref{S:qk} below).  This
poses additional combinatorial difficulties for proving our Pieri
formula.

For complete flag manifolds $G/B$, formulas for similar
purposes have been conjectured by Lenart and Maeno
\cite{lenart.maeno:quantum} and by Lenart and Postnikov
\cite{lenart.postnikov:affine}, based on a combinatorial point of
view.  However, due to the lack of functoriality in quantum
cohomology, the connection to our results is unclear.

This paper is organized as follows.  In section~\ref{S:gwvar} we prove
that Gromov-Witten varieties of high degree for Grassmannians of type
A are rational.  Section \ref{S:ktheory} sets up notation for
equivariant $K$-theory and proves a criterion ensuring that the
pushforward of the Grothendieck class of a variety is equal to the
Grothendieck class of its image.  Section~\ref{S:gwinv} contains a
general discussion of Gromov-Witten invariants of various types, and
proves our formula for equivariant $K$-theoretic Gromov-Witten
invariants of Grassmannians.  In Section~\ref{S:qk} we define the
quantum $K$-theory ring of a Grassmann variety, state our Pieri
formula, derive a Giambelli type formula, prove the $S_3$-symmetry
property for the structure constants, and dualize the Schubert basis.
We also discuss consequences and applications to computing structure
constants and $K$-theoretic Gromov-Witten invariants, and provide the
multiplication tables for the equivariant quantum $K$-theory rings of
$\P^1$ and $\P^2$.  Section~\ref{S:proof} contains the proof of our
Pieri formula, and section~\ref{S:comin} generalizes our formula for
Gromov-Witten invariants to cominuscule spaces.

\subsection{Acknowledgments}

This paper continues the above mentioned work with Andrew Kresch and
Harry Tamvakis; we are grateful for the ideas and insights that have
undoubtedly floated between the two projects.  During the preparation
of this paper we also benefited from helpful discussions and comments
of many other people, including Arend Bayer, Prakash Belkale, Izzet
Coskun, Ionut Ciocan-Fontanine, Johan de Jong, Friedrich Knop, Shrawan
Kumar, and Chris Woodward.  Special thanks are due to Ravi Vakil who
made valuable comments on an earlier proof of the rationality of
Gromov-Witten varieties.  The second author thanks the Max Planck
Institut f\"ur Mathematik for a providing a stimulating environment
during the early stages of this work.


\section{Rationality of Gromov-Witten varieties} \label{S:gwvar}

Let $X = G/P$ be a homogeneous space defined by a complex connected
semi-simple linear algebraic group $G$ and a parabolic subgroup $P$.  An
$N$-pointed {\em stable map\/} (of genus zero) to $X$ is a morphism of
varieties $f : C \to X$, where $C$ is a tree of projective lines,
together with $N$ distinct non-singular marked points of $C$, ordered
from 1 to $N$, such that any component of $C$ that is mapped to a
single point in $X$ contains at least three special points, where
special means marked or singular \cite{kontsevich:enumeration}.  The
{\em degree\/} of $f$ is the homology class $f_* [C] \in H_2(X;\Z)$.

A fundamental tool in Gromov-Witten theory is Kontsevich's moduli
space $\Mb_{0,N}(X,d)$, which parametrizes all $N$-pointed stable maps
to $X$ of degree $d$.  This space is equipped with {\em evaluation
  maps\/} $\ev_i : \Mb_{0,N}(X,d) \to X$ for $1 \leq i \leq N$, where
$\ev_i$ sends a stable map $f$ to its image of the $i$-th marked point
in its domain $C$.  We let $\ev = \ev_1 \times \dots \times \ev_N :
\Mb_{0,N} \to X^N := X \times \dots \times X$ denote the total
evaluation map.  For $N\geq 3$, there is also a forgetful map $\rho :
\Mb_{0,N}(X,d) \to \Mb_{0,N} := \Mb_{0,N}(\pt,0)$ which sends a stable
map to its domain (after collapsing unstable components).  The
(coarse) moduli space $\Mb_{0,N}(X,d)$ is a normal projective variety
with at worst finite quotient singularities, and its dimension is
given by
\begin{equation}
  \dim \Mb_{0,N}(X,d) = \dim(X) + \int_d c_1(T_X) + N-3 \,.
\end{equation}
We refer to the notes \cite{fulton.pandharipande:notes} for the
construction of this space.  It has been proved by Kim and
Pandharipande \cite{kim.pandharipande:connectedness} and by Thomsen
\cite{thomsen:irreducibility} that the Kontsevich space
$\Mb_{0,N}(X,d)$ is irreducible.  Kim and Pandharipande also showed that
$\Mb_{0,N}(X,d)$ is rational.

Recall that a {\em Schubert variety\/} in $X$ is an orbit closure for
the action of a Borel subgroup $B \subset G$.  A collection of
Schubert varieties $\Omega_1, \dots, \Omega_N$ can be moved in {\em general
  position\/} by translating them with general elements of $G$.  Given
such a collection and a degree $d \in H_2(X)$, there is a {\em
  Gromov-Witten variety\/} defined by
\begin{equation}
  GW_d(\Omega_1, \dots, \Omega_N) = \ev^{-1}(\Omega_1 \times \dots
  \times \Omega_N) \ \subset \Mb_{0,N}(X,d) \,.
\end{equation}
When this intersection is finite, its cardinality is the {\em
  Gromov-Witten invariant\/} associated to the corresponding Schubert
classes.  In general, its Euler characteristic defines a {\em
  $K$-theoretic Gromov-Witten invariant} \cite{lee:quantum}.

Lee and Pandharipande asked which Gromov-Witten varieties are rational
in \cite{lee.pandharipande:reconstruction}.  They also announced that
for a fixed degree $d \in H_2(\bP^1)$, the Gromov-Witten variety
$GW_d(P_1,\dots,P_N) \subset \Mb_{0,N}(\bP^1,d)$ is rational for only
finitely many integers $N$, where $P_1,\dots,P_N \in \bP^1$ are
general points.  Pandharipande had earlier shown that the variety
$GW_d(P_1,\dots,P_{3d-2}) \subset \Mb_{0,3d-2}(\bP^2,d)$ is a
non-singular curve of positive genus for $d \geq 3$
\cite{pandharipande:canonical}.  In preparation for our computation of
$K$-theoretic Gromov-Witten invariants in section~\ref{S:gwinv} we
will prove that, if $X$ is a Grassmannian of type A and $N \geq 3$ is
a fixed integer, then $GW_d(\Omega_1,\dots,\Omega_N) \subset
\Mb_{0,N}(X,d)$ is rational for all but finitely many degrees $d$.
Our formulas for Gromov-Witten invariants can in turn be used to
locate non-rational 3-pointed Gromov-Witten varieties of positive
dimension, see Example~\ref{E:nonrat3pt}.

Let $X = \Gr(m,n) = \{ V \subset \C^n : \dim V = m \}$ be the
Grassmannian of $m$-dimensional vector subspaces of $\C^n$.  This
variety has a tautological subbundle $\cS \subset \cO_X^{\oplus n} = X
\times \C^n$ given by $\cS = \{ (V,u) \in X \times \C^n : u \in V \}$.
Let $e_1,\dots,e_n$ be the standard basis of $\C^n$.  Notice that
$H_2(X) = \Z$, so the degree of a stable map to $X$ can be identified
with a non-negative integer.

\begin{thm} \label{T:gwrat}
  Let $P \in X^N$ and $Q \in \Mb_{0,N}$ be general points, with $N\geq
  3$, and let $d \geq 0$.  Then the intersection $\ev^{-1}(P) \bigcap
  \rho^{-1}(Q) \subset \Mb_{0,N}(X,d)$ is either empty or an
  irreducible rational variety.
\end{thm}
\begin{proof}
  Since the point $Q \in \Mb_{0,N}$ is general, it consists of $N$
  distinct points $(x_1\cc y_1), \dots, (x_N\cc y_N)$ in $\bP^1$.
  Write $d = mp+r$ where $0 \leq r < m$, and set $s = m-r$.  Define
  the vector bundle $\cE = \cO(-p)^{\oplus s} \oplus \cO(-p-1)^{\oplus
    r}$ on $\P^1$.  Then the inverse image $\rho^{-1}(Q) \subset
  \Mb_{0,N}(X,d)$ has a dense open subset of maps $f : \bP^1 \to X$
  for which $f^*(\cS) \cong \cE$.  Any map in this subset is given by
  an injective element in $\Hom_{\P^1}(\cE, \cO^{\oplus n})$, and two
  injective elements define the same map if and only if they differ by
  an automorphism of $\cE$.  The marked points in the domain of each
  map is given by $Q$.
  
  We have $\Hom_{\P^1}(\cE,\cO^{\oplus n}) = (\C^n)^{s(p+1)} \oplus
  (\C^n)^{r(p+2)} = \{(u_{ij},v_{ij})\}$, where $u_{ij} \in \C^n$ is
  defined for $1 \leq i \leq s$ and $0 \leq j \leq p$, and $v_{ij} \in
  \C^n$ is defined for $1 \leq i \leq r$ and $0 \leq j \leq p+1$.  An
  injective element $(u_{ij},v_{ij})$ defines the map $f : \P^1 \to X$
  which sends $(x\!:\!y) \in \P^1$ to the span of the vectors
  $\sum_{j=0}^p x^j y^{p-j}\, u_{ij}$ for $1 \leq i \leq s$ and
  $\sum_{j=0}^{p+1} x^j y^{p+1-j}\, v_{ij}$ for $1 \leq i \leq r$. The
  automorphisms of $\cE$ are given by $\Aut(\cE) = \GL(s) \times
  \GL(r) \times \Mat(r,s) \times \Mat(r,s)$, and composition with the
  element $(a,b,c,c') \in \Aut(\cE)$ is given by $(a,b,c,c') .
  (u_{ij},v_{ij}) = (u'_{ij},v'_{ij})$ where $u'_{ij} = \sum_{l=1}^s
  a_{il} u_{lj}$ and $v'_{ij} = \sum_{l=1}^{r} b_{il} v_{lj} +
  \sum_{l=1}^{s} c_{il} u_{lj} + \sum_{l=1}^s c'_{il} u_{l,j-1}$.
  Here we write $u_{ij} = 0$ if $j \in \{-1,p+1\}$.
  
  It follows from this description that a dense open subset $U$ of the
  injective elements of $\Hom_{\P^1}(\cE, \cO^{\oplus n})$ modulo the
  action of $\Aut(\cE)$ have unique representatives of the form
  $(u_{ij},v_{ij})$ where $u_{ip} = e_i + \wt u_{ip}$ with $\wt u_{ip}
  \in 0^s \oplus \C^{n-s}$; $v_{i,p+1} = e_{s+i} + \wt v_{i,p+1}$ with
  $\wt v_{i,p+1} \in 0^m \oplus \C^{n-m}$; and $v_{ip} \in 0^s \oplus
  \C^{n-s}$.  This shows that $U \subset \rho^{-1}(Q)$ is isomorphic
  to an open subset of an affine space, so $\rho^{-1}(Q)$ is rational.
  The evaluation maps $\ev_k : U \to X$ are then given by
  $\ev_k(u_{ij},v_{ij}) = f(x_k\cc y_k)$ where $f : \bP^1 \to X$ is
  defined as above.
  
  It follows from Kleiman's transversality theorem
  \cite[Thm.~2]{kleiman:transversality} that all components of
  $\ev^{-1}(P) \bigcap \rho^{-1}(Q)$ meet $U$.  If we write $P =
  (V_1,\dots,V_N)$ with $V_i \in X$, then we conclude that
  $\ev^{-1}(P) \bigcap \rho^{-1}(Q)$ is birational to the set of
  points $(u_{ij},v_{ij}) \in U$ which satisfy that $f(x_k\cc y_k) =
  V_k$ for each $k$.  In other words, we require that $\sum_{j=0}^p
  x_k^j y_k^{p-j} u_{ij} \in V_k$ for $1 \leq i \leq s$ and
  $\sum_{j=0}^{p+1} x_k^j y_k^{p+1-j} v_{ij} \in V_k$ for $1 \leq i
  \leq r$.  Since this amounts to a set of affine equations on
  $\{(u_{ij}, v_{ij})\}$, we conclude that $\ev^{-1}(P) \bigcap
  \rho^{-1}(Q)$ is either empty or rational, as claimed.
\end{proof}

The following corollary will be used to compute $K$-theoretic
Gromov-Witten invariants of large degrees.

\begin{cor} \label{C:ratA}
  Let $V_1, V_2, V_3 \in X$ be general points.  Then $GW_d(V_1, V_2,
  V_3)$ is rational for all degrees $d \geq \max(m,n-m)$.
\end{cor}
\begin{proof}
  Since $\Mb_{0,3}$ has only one point, the intersection of
  Theorem~\ref{T:gwrat} is equal to $GW_d(V_1,V_2,V_3)$.  We must show
  that this variety is not empty.  We may assume that $m \leq n-m \leq
  d$.  Then the dimension of the vector space $W = V_1 + V_2 \subset
  \C^n$ is $2m$.  Choose $V'_3 \in \Gr(m,W) \subset \Gr(m,\C^n)$ such
  that $V'_3 \cap V_1 = V'_3 \cap V_2 = 0$ and $V'_3 \cap V_3 = W \cap
  V_3$.  Using e.g.\ \cite[Prop.~1]{buch.kresch.ea:gromov-witten} we
  can find a rational map $f_1 : \bP^1 \to \Gr(m,W)$ of degree $m$
  such that $V_1$, $V_2$, and $V'_3$ are contained in the image of
  $f_1$.  Note that $A = V_3 \cap V'_3$ has dimension $m-d_2$ and $B
  = V_3 + V'_3$ has dimension $m+d_2$ where $d_2 = \min(n-2m,m)$.
  Another application of \cite[Prop.~1]{buch.kresch.ea:gromov-witten}
  now shows that $V_3$ and $V'_3$ are contained in the image of a
  rational map $f_2 : \bP^1 \to \Gr(d_2,B/A) \subset X$ of degree
  $d_2$.  If we let $C$ be the union of the domains of $f_1$ and
  $f_2$, with the points mapping to $V'_3$ identified, then $f_1$ and
  $f_2$ define a stable map $C \to X$ of degree $m+d_2 \leq n-m \leq
  d$ whose image contains $V_1$, $V_2$, and $V_3$.  If necessary, we
  can add extra components to $C$ to obtain a stable map of degree
  $d$.  This constructs a point of $GW_d(V_1,V_2,V_3)$ and finishes the
  proof.
\end{proof}

We finally give a ``converse'' to Lee and Pandharipande's announcement
from \cite[p.~1379]{lee.pandharipande:reconstruction}.  Notice that
for sufficiently large degrees $d$, the map $(\ev, \rho) :
\Mb_{0,N}(X,d) \to X^N \times \Mb_{0,N}$ is surjective.
Theorem~\ref{T:gwrat} implies that its fibers are irreducible for all
points in a dense subset of $X^N \times \Mb_{0,N}$, and using the
Stein factorization we deduce that all fibers are connected.

\begin{cor}
  Let $\Omega_1, \Omega_2, \dots, \Omega_N \subset X$ be Schubert
  varieties in general position.  Choose $d$ large enough so that the
  map $\Mb_{0,N}(X,d) \to X^N \times \Mb_{0,N}$ is surjective.  Then
  the Gromov-Witten variety $GW_d(\Omega_1, \dots, \Omega_N)$ is
  birational to an affine bundle over $\Omega_1 \times \dots \times
  \Omega_N \times \Mb_{0,N}$.  In particular,
  $GW_d(\Omega_1,\dots,\Omega_N)$ is rational.
\end{cor}
\begin{proof}
  We first show that $GW_d(\Omega_1, \dots, \Omega_N)$ is an
  irreducible variety.  The Kleiman-Bertini theorem
  \cite[Remark~7]{kleiman:transversality} implies that this variety is
  locally irreducible, so it suffices to prove that it is connected.
  This follows because all fibers of the proper surjective map
  $(\ev,\rho) : GW_d(\Omega_1,\dots,\Omega_N) \to \Omega_1
  \times\dots\times \Omega_N \times \Mb_{0,N}$ are connected.
  Finally, to see that $GW_d(\Omega_1,\dots,\Omega_N)$ is birational
  to an affine bundle, we observe that the construction used to prove
  Theorem~2.1 also parametrizes an open subset of $\Mb_{0,N}(X,d)$ in
  terms of local coordinates on $X^N \times \Mb_{0,N}$.
\end{proof}

\begin{remark} \label{R:unirat}
  In section \ref{S:comin} we will examine $K$-theoretic Gromov-Witten
  invariants for cominuscule homogeneous spaces, at which point a
  cominuscule analogue of Corollary~\ref{C:ratA} would be desirable.
  We have work in progress showing that certain Gromov-Witten
  varieties are unirational, which suffices for our purposes.  This
  includes 3-points Gromov-Witten varieties for maximal orthogonal
  Grassmannians and 2-point Gromov-Witten varieties for Lagrangian
  Grassmannians.  We can also show that 3-point Gromov-Witten
  varieties for Lagrangian Grassmannians have a rational component,
  but not that they are irreducible.  These developments will be
  explained elsewhere.  
\ignore{
  For three general points
  $V_1,V_2,V_3$ in $X = \OG(2,5)$, one can show that the variety
  $GW_3(V_1,V_2,V_3)$ is birational to the hypersurface $\Spec
  \C[a,b,c,d,e,f,g]/(P)$, where the polynomial $P$ is defined by:
  \[\begin{split}
    P =& -2a^2d+2fe+ad^2+dab^2+2ca^2d-edb+edc^2+f^2ab+a^2df-2cfa \\
       & -2efa+a^2bf+2bfc-bfe+b^2fa-bfc^2-4adf+adc^2+a^2de-2a^2db \\
       & +2dab-2cad+d^2ab-2ecf-2cd^2a+2dfe-f-2dfab-3deab+3f^2 \\
       & +2ecad-2bcad+fabe+2fa+ad-ed^2+3de^2-fe^2+a^2d^2-2cf^2 \\
       & +a^3d-fc^2-ed-f^2a+2adcf+2dcf-f^2e-bf+2cf-df-fa^2
    \end{split}
  \]
  It would be interesting to know if this hypersurface is rational.
}
\end{remark}



\section{Direct images of Grothendieck classes} \label{S:ktheory}

In this section we prove some facts about equivariant $K$-theory in
preparation for our computation of Gromov-Witten invariants.  Our main
references are Chapter 5 in \cite{chriss.ginzburg:representation} and
Section 15.1 in \cite{fulton:intersection}.  To honor the assumptions
in the first reference, we will assume that all varieties are
quasi-projective over $\C$.

Let $G$ be a complex linear algebraic group and let $X$ be a
(quasi-projective) $G$-variety.  An {\em equivariant sheaf\/} on $X$
is a coherent $\cO_X$-module $\cF$ together with a given isomorphism
$I : a^*\cF \cong p_X^*\cF$, where $a : G \times X \to X$ is the
action and $p_X : G \times X \to X$ the projection.  This isomorphism
must satisfy that $(m \times \id_X)^* I = p_{23}^*I \circ (\id_G
\times a)^*I$ as morphisms of sheaves on $G \times G \times G$, where
$m$ is the group operation on $G$ and $p_{23}$ is the projection to
the last two factors of $G \times G \times X$.

The {\em equivariant $K$-homology group\/} $K_G(X)$ is the
Grothendieck group of equivariant sheaves on $X$, i.e.\ the free
Abelian group generated by isomorphism classes $[\cF]$ of equivariant
sheaves, modulo relations saying that $[\cF] = [\cF'] + [\cF'']$ if
there exists an equivariant exact sequence $0 \to \cF' \to \cF \to
\cF'' \to 0$.  This group is a module over the {\em equivariant
  $K$-cohomology ring\/} $K^G(X)$, defined as the Grothendieck group
of equivariant vector bundles on $X$.  Both the multiplicative
structure of $K^G(X)$ and the module structure of $K_G(X)$ are given
by tensor products.  The {\em Grothendieck class\/} of $X$ is the
class $[\cO_X] \in K^G(X)$ of its structure sheaf.  If $X$ is
non-singular, then the implicit map $K^G(X) \to K_G(X)$ that sends a
vector bundle to its sheaf of sections is an isomorphism; this follows
because every equivariant sheaf on $X$ has a finite resolution by
equivariant vector bundles
\cite[5.1.28]{chriss.ginzburg:representation}.

Given an equivariant morphism of $G$-varieties $f : X \to Y$, there is
a ring homomorphism $f^* : K^G(Y) \to K^G(X)$ defined by pullback of
vector bundles.  If $f$ is proper, then there is also a pushforward
map $f_* : K_G(X) \to K_G(Y)$ defined by $f_*[\cF] = \sum_{i \geq 0}
(-1)^i [R^i f_* \cF]$.  This map is a homomorphism of $K^G(Y)$-modules
by the projection formula \cite[Ex.~III.8.3]{hartshorne:algebraic*1}.
The Godement resolution can be used to obtain equivariant structures
on the higher direct image sheaves of $\cF$.  Both pullback and
pushforward are functorial with respect to composition of morphisms.

Recall that the variety $X$ has {\em rational singularities\/} if
there exists a desingularization $\pi : \wt X \to X$ for which $\pi_*
\cO_{\wt X} = \cO_X$ and $R^i \pi_* \cO_{\wt X} = 0$ for $i>0$.  When
this is true, these identities hold for any desingularization, and $X$
is normal (see e.g the proof of
\cite[Cor.~11.4]{hartshorne:algebraic*1}).  If $X$ is a
$G$-variety, then it follows from
\cite[Thm.~7.6.1]{villamayor-u:patching} or
\cite[Thm.~13.2]{bierstone.milman:canonical} that $X$ has an {\em
  equivariant desingularization}.  More precisely, there exists an
equivariant projective birational morphism $\pi : \wt X \to X$ from a
non-singular $G$-variety $\wt X$ (see e.g.\ \cite[\S
4]{reichstein.youssin:equivariant}).  This implies that
$\pi_*[\cO_{\wt X}] = [\cO_X] \in K_G(X)$.  More generally, if $f : X
\to Y$ is any equivariant proper birational map of $G$-varieties with
rational singularities, then the composition $f \pi : \wt X \to Y$ is
a desingularization of $Y$, so we obtain $f_* [\cO_X] = f_* \pi_*
[\cO_{\wt X}] = [\cO_Y] \in K_G(Y)$ by functoriality.  We need the
following generalization.

\begin{thm} \label{T:push}
  Let $f : X \to Y$ be a surjective equivariant map of projective
  $G$-varieties with rational singularities.  Assume that the general
  fiber of $f$ is rational, i.e.\ $f^{-1}(y)$ is an irreducible
  rational variety for all closed points in a dense open subset of
  $Y$.  Then $f_*[\cO_X] = [\cO_Y] \in K_G(Y)$.
\end{thm}

We will deduce this statement from the following result of Koll\'ar
\cite[Thm.~7.1]{kollar:higher}.

\begin{thm}[Koll\'ar] \label{T:kollar}
  Let $\psi : X \to Y$ be a surjective map between projective
  varieties, with $X$ smooth and $Y$ normal.  Assume that the
  geometric generic fiber $F = X \times_Y \Spec \wb{\C(Y)}$ is
  connected.  Then the following are equivalent:

(i) $R^i \psi_* \cO_X = 0$ for all $i>0$;

(ii) $Y$ has rational singularities and $H^i(F,\cO_F) = 0$ for all
$i>0$.
\end{thm}

To show that Koll\'ar's theorem applies to our situation, we need the
following two lemmas; they are most likely known, but since we lack a
reference, we supply their proofs.

\begin{lemma} \label{L:connected} Let $\varphi : X \to Y$ be a
  dominant morphism of irreducible varieties, and assume that
  $\varphi^{-1}(y)$ is connected for all closed points $y$ in a dense
  open subset of $Y$.  Then the geometric generic fiber $X \times_Y
  \Spec \wb{\C(Y)}$ is connected.
\end{lemma}
\begin{proof}
  We may assume that $X = \Spec(S)$ and $Y = \Spec(R)$ are both
  affine.  If the geometric generic fiber is disconnected, then there
  exists non-zero elements $f, g \in S \otimes_R \wb{K(R)}$ such that
  $f+g=1$, $f^2=f$, $g^2=g$, and $fg=0$ (cf.\ 
  \cite[Ex.~2.19]{hartshorne:algebraic*1}).  These elements $f$ and
  $g$ will be contained in $S \otimes_R R'$ for some finitely
  generated $\C$-algebra $R'$ with $R \subset R' \subset \wb{K(R)}$.
  Now consider the diagram:
  \[\xymatrix{ 
    \Spec(S) \ar[r]^{\varphi} & \Spec(R) \\
    \Spec(S \otimes_R R')\ar[u] \ar[r] & \Spec(R')\ar[u] }
  \]
  
  By Grothendieck's generic freeness lemma, we may assume that $S =
  \bigoplus_i R s_i$ is a free $R$-module, after replacing $R$ with a
  localization $R_h$.  Write $f = \sum s_i \otimes f_i$ and $g = \sum
  s_i \otimes g_i$ with $f_i, g_i \in R'$, and choose $i$ and $j$ such
  that $h = f_i g_j \in R'$ is non-zero.  Then the images of $f$ and
  $g$ in $S \otimes_R R'/P$ are non-zero for each (closed) point $P
  \in \Spec(R'_h)$, which implies that the fiber $\varphi^{-1}(P \cap
  R) = \Spec(S \otimes_R R'/P)$ is disconnected.  This is a
  contradiction because $\Spec(R'_h) \to \Spec(R)$ is a dominant
  morphism, so its image contains a dense open subset of $\Spec(R)$.
\end{proof}

We note that Lemma~\ref{L:connected} (and its proof) is valid for
varieties over any algebraically closed field.  The same argument also
shows that, if the general fiber of $\varphi$ is integral, then so is
the geometric generic fiber.

\begin{lemma} \label{L:push}
  Let $f : X \to Y$ be a surjective projective morphism of irreducible
  varieties of characteristic zero, with $Y$ normal.  Assume that
  $f^{-1}(y)$ is connected for all closed points $y$ in a dense open
  subset of $Y$.  Then $f_*(\cO_X) = \cO_Y$.
\end{lemma}
\begin{proof}
  By the proof of \cite[III.11.5]{hartshorne:algebraic*1}, $f$ has a
  Stein factorization $f = g f'$, where $f' : X \to Y'$ is projective
  with $f'_* \cO_X = \cO_{Y'}$ and $g : Y' \to Y$ is finite.  Since
  the general fiber of $f$ is connected and the characteristic is
  zero, the map $g : Y' \to Y$ must be birational.  But then $g$ is an
  isomorphism by Zariski's Main Theorem.
\end{proof}

\begin{proof}[Proof of Theorem~\ref{T:push}]
  Let $\pi : \wt X \to X$ be an equivariant projective
  desingularization of $X$.  Since $X$ has rational singularities we
  know that $\pi_* \cO_{\wt X} = \cO_X$ and $\pi_*[\cO_{\wt X}] =
  [\cO_X]$, so it is enough to show that $\psi_*[\cO_{\wt X}] =
  [\cO_Y]$ where $\psi = f \pi$.  Lemma \ref{L:push} implies that $f_*
  \cO_X = \cO_Y$, so $\psi_* \cO_{\wt X} = \cO_Y$.  By Koll\'ar's
  Theorem~\ref{T:kollar}, it is therefore enough to prove that the
  geometric generic fiber $F = \wt X \times_Y \Spec \wb{\C(Y)}$ is
  connected and $H^i(F,\cO_F) = 0$ for $i>0$.
  
  By \cite[III.10.7]{hartshorne:algebraic*1} we can find a dense open
  subset $U \subset Y$ such that $\psi : \psi^{-1}(U) \to U$ is
  smooth.  Since the fibers of $\psi$ are connected by
  \cite[III.11.3]{hartshorne:algebraic*1}, it follows that
  $\psi^{-1}(y)$ is a non-singular rational projective variety for
  every closed point $y$ in a dense open subset of $Y$.  In
  particular, Lemma~\ref{L:connected} implies that $F$ is
  connected.
  
  To obtain the vanishing of cohomology, let $y \in Y$ be a closed
  point such that $\wt X_y = \psi^{-1}(y)$ is a non-singular rational
  projective variety.  Then we have $H^0(\wt X_y, \cO_{\wt X_y}) = \C$
  and $H^i(\wt X_y, \cO_{\wt X_y}) = 0$ for all $i > 0$
  \cite[p.~494]{griffiths.harris:principles}.  It now follows from
  \cite[III.12.11]{hartshorne:algebraic*1} that $R^i \psi_*(\cO_{\wt
    X})$ is zero in an open neighborhood of $y$, and that $H^i(\wt
  X_z, \cO_{\wt X_z}) = 0$ for all points $z$ in this neighborhood,
  for $i>0$.  Taking $z$ to be the generic point of $Y$, we obtain
  $H^i(F',\cO_{F'}) = 0$ for $i>0$, where $F' = \wt X \times_Y \Spec
  \C(Y)$.  Finally, since $\Spec \wb{\C(Y)} \to \Spec \C(Y)$ is a flat
  morphism, it follows from \cite[III.9.3]{hartshorne:algebraic*1}
  that $H^i(F,\cO_F) = 0$ for $i>0$, as required.
\end{proof}

We also need the following consequence of the projection formula.
Recall that $K^G(X)$ and $K_G(X)$ can be identified if $X$ is
non-singular.

\begin{lemma} \label{L:product}
  Let $q_i : X_i \to Y$ be flat proper equivariant maps of
  non-singular $G$-varieties for $1 \leq i \leq n$, and let $\alpha_i
  \in K^G(X_i)$.  Set $P = X_1 \times_Y \dots \times_Y X_n$ with
  diagonal $G$-action and projections $e_i : P \to X_i$, and set $\psi
  = q_i e_i : P \to Y$.  Then $\psi_*(e_1^*\alpha_1 \cdot
  e_2^*\alpha_2 \cdots e_n^* \alpha_n) = q_{1,*}\alpha_1 \cdot
  q_{2,*}\alpha_2 \cdots q_{n,*}\alpha_n \in K^G(Y)$.
\end{lemma}
\begin{proof}
  Let $P' = X_1 \times_Y \dots \times_Y X_{n-1}$ with projections
  $e'_i : P' \to X_i$ and set $\psi' = q_i e'_i$.  Then we have a
  fiber square with flat, proper, $G$-equivariant maps:
  \[ \xymatrix{
     P \ar[r]^{e_n} \ar[d]_{q'_n} & X_n \ar[d]^{q_n} \\
     P' \ar[r]_{\psi'} & Y }
  \]
  It follows from \cite[III.9.3]{hartshorne:algebraic*1} or
  \cite[5.3.15]{chriss.ginzburg:representation} that $q'_{n,*}
  e_n^*\alpha_n = {\psi'}^* q_{n,*} \alpha_n \in K^G(P')$.  By
  induction on $n$ we therefore obtain
  \[\begin{split}
    \psi_*(e_1^*\alpha_1 \cdots e_n^*\alpha_n)
    &= \psi'_* q'_{n,*}( {q'}_n^*({e'}_1^* \alpha_1 \cdots {e'}_{n-1}^*
       \alpha_{n-1}) \cdot e_n^* \alpha_n) \\
    &= \psi'_*( {e'}_1^*\alpha_1 \cdots {e'}_{n-1}^*\alpha_{n-1} \cdot
        q'_{n,*} e_n^*\alpha_n) \\
    &= \psi'_*( {e'}_1^*\alpha_1 \cdots {e'}_{n-1}^*\alpha_{n-1} \cdot
       {\psi'}^* q_{n,*} \alpha_n) \\
    &= \psi'_*( {e'}_1^*\alpha_1 \cdots {e'}_{n-1}^*\alpha_{n-1}) \cdot
       q_{n,*} \alpha_n \\
    &= q_{1,*}\alpha_1 \cdots q_{n-1,*}\alpha_{n-1} \cdot q_{n,*}
       \alpha_n
  \end{split}\]
  as required.
\end{proof}

We finally need the following facts about the pushforward and pullback
of Schubert classes between homogeneous spaces.  Let $G$ be a complex
connected semisimple linear algebraic group, $T$ a maximal torus, and
$P$ and $Q$ Borel subgroups, such that $T \subset P \subset Q \subset
G$.  Let $f : G/Q \to G/P$ be the projection.

\begin{lemma} \label{L:pushschub}
  {\rm (a)} If $\Omega \subset G/Q$ is any $T$-stable Schubert
  variety, then $f_*([\cO_\Omega]) = [\cO_{f(\Omega)}] \in K_T(G/P)$.
  {\rm (b)} If $\Omega \subset G/P$ is any $T$-stable Schubert
  variety, then $f^*([\cO_\Omega]) = [\cO_{f^{-1}(\Omega)}] \in
  K^T(G/P)$.
\end{lemma}

Part (a) of this lemma is known from
\cite[Thm.3.3.4(a)]{brion.kumar:frobenius}, and part (b) is true
because $f$ is a flat morphism.


\section{Gromov-Witten invariants} \label{S:gwinv}

\subsection{Definitions}

We start our discussion of Gromov-Witten invariants by recalling the
definitions and some general facts.  Let $X = G/P$ be a homogeneous
space and let $d \in H_2(X)$ be a degree.  Given $K$-theory classes
$\alpha_1, \dots, \alpha_N \in K^\circ(X)$, Lee and Givental define a
$K$-theoretic Gromov-Witten invariant by \cite{lee:quantum,
  givental:on}
\begin{equation} 
  I_d(\alpha_1,\dots,\alpha_N) \ = \ 
  \chi(\ev_1^*(\alpha_1) \cdots \ev_N^*(\alpha_N)) \ \in \Z \,,
\end{equation}
where $\chi$ denotes Euler characteristic, i.e.\ proper pushforward
along the structure morphism $\rho : \Mb_{0,N}(X,d) \to \{\pt\}$.  If
the classes $\alpha_i$ are structure sheaves of closed subvarieties
$\Omega_1,\dots,\Omega_N \subset X$ in general position such that
$\sum \codim(\Omega_i) = \dim \Mb_{0,N}(X,d)$, then the invariant
$I_d(\cO_{\Omega_1}, \dots, \cO_{\Omega_d})$ is equal to the
cohomological invariant
\[ I_d([\Omega_1], \dots, [\Omega_d]) = \int_{\Mb_{0,N}(X,d)}
   \ev_1^*[\Omega_1] \cdots \ev_N^*[\Omega_N] \,,
\]
which in turn is equal to the number $\#
GW_d(\Omega_1,\dots,\Omega_N)$ of stable maps $C \to X$ of degree $d$
for which the $i$-th marked point of $C$ is mapped into $\Omega_i$
(see \cite[Lemma~14]{fulton.pandharipande:notes}).  In general, the
$K$-theoretic invariant $I_d(\cO_{\Omega_1}, \dots, \cO_{\Omega_N})$
is equal to the Euler characteristic of the structure sheaf of the
Gromov-Witten variety $GW_d(\Omega_1,\dots,\Omega_N)$, which is the
definition of $K$-theoretic Gromov-Witten invariants used in
\cite{lee.pandharipande:reconstruction}.  This can be seen by applying
the $K$-theoretic (relative) Kleiman-Bertini theorem of Sierra
\cite[Thm.~2.2]{sierra:general} (see also
\cite{miller.speyer:kleiman-bertini} for the non-relative case).  Let
$G^N = G \times G \times \dots \times G$ act componentwise on $X^N$
and consider the diagram
\[\xymatrix{
  \Mb = \Mb_{0,N}(X,d) \ar[r]^{\ \ \ \ \ \ \ \ev} & X^N \\
  & G^N \times W \ar[u]^{f} \ar[r]^{\ \ \pi} & G^N}
\]
where $W = \Omega_1 \times \dots \times \Omega_N \subset X^N$ and $f$
is defined by the action.  If we let $G^N$ act on the first factor of
$G^N \times W$, then the map $f$ is equivariant, and therefore flat
since the action on $X^N$ is transitive.  Sierra's theorem implies
that $\text{\sl Tor}_i^{X^N}(\cO_{\Mb}, \cO_{g.W}) = 0$ for some $g
\in G^N$ and all $i>0$.  Replacing $W$ with $g.W$, we obtain (cf.\ 
\cite[Ex.~15.1.8]{fulton:intersection})
\[ {\textstyle \prod_{i=1}^N} \ev_i^*[\cO_{\Omega_i}] = \ev^*[\cO_W] = 
   [\cO_{\Mb} \otimes_{\cO_{X^N}} \cO_W] = 
   [\cO_{GW_d(\Omega_1,\dots,\Omega_N)}]
\]
in $K_\circ(\Mb)$, which finally implies that
$I_d(\cO_{\Omega_1},\dots,\cO_{\Omega_N}) =
\chi(\cO_{GW_d(\Omega_1,\dots,\Omega_n)})$.  Notice that in contrast
to cohomological Gromov-Witten invariants, the $K$-theoretic
invariants need not vanish for large degrees $d \in H_2(X)$.

More generally, fix a maximal torus $T \subset G$.  Then $T$ acts on
$X$ and $\Mb_{0,N}(X,d)$, and the evaluation maps are equivariant.
The {\em equivariant $K$-theoretic Gromov-Witten invariant\/} given by
classes $\alpha_1, \dots, \alpha_N \in K^T(X)$ is defined as the
virtual representation
\begin{equation}
  I_d^T(\alpha_1,\dots,\alpha_N) \ = \ \euler{\Mb}^T(\ev_1^*(\alpha_1) 
  \cdots \ev_N^*(\alpha_N)) \ \in K^T(\pt) \,,
\end{equation}
where $\euler{\Mb}^T$ is the equivariant pushforward along $\rho$.  If
we write $T = (\C^*)^n$, then the virtual representations of $T$ form
the Laurent polynomial ring $K^T(\pt) = \Z[L_1^{\pm 1}, \dots,
L_n^{\pm 1}]$, where $L_i \cong \C$ is the representation with
character $(t_1,\dots,t_n) \mapsto t_i$.  This ring is contained in
the ring of formal power series $\Z\llbracket y_1,\dots,y_n
\rrbracket$ in the variables $y_i = 1 - L_i^{-1} \in K^T(\pt)$.  If we
misuse notation and write $y_i$ also for the equivariant Chern class
of $L_i$, then the $T$-equivariant cohomology ring of a point is the
ring $H^*_T(\pt) = \Z[y_1, \dots, y_n]$.

The study of equivariant Gromov-Witten invariants was pioneered by
Givental \cite{givental:equivariant} (see also
\cite{givental.kim:quantum}), who defined the cohomological invariants
\begin{equation*}
  I_d^T(\beta_1,\dots,\beta_N)
  \ = \ 
  \int^T_{\Mb_{0,3}(X,d)} \ev_1^*(\beta_1) \cdots \ev_N^*(\beta_N) 
  \ := \ 
  \rho_*^T(\ev_1^*(\beta_1) \cdots \ev_N^*(\beta_N))
\end{equation*}
in $H^*_T(\pt)$ for equivariant cohomology classes
$\beta_1,\dots,\beta_N \in H^*_T(X)$.  These invariants can be
obtained as the leading terms of $K$-theoretic invariants as follows.
For any $T$-variety $Y$ we let $\ch_T : K^T(Y) \to \wh H_T(Y) :=
\prod_{i=0}^\infty H^{2i}_T(Y;\Q)$ be the equivariant Chern character,
see \cite[Def.~3.1]{edidin.graham:riemann-roch}.  By the equivariant
Hirzebruch formula \cite[Cor.~3.1]{edidin.graham:riemann-roch} we then
have
\begin{equation} \label{E:ekgw_lowest}
  \ch_T\left( I_d^T(\alpha_1,\dots,\alpha_N) \right) =
  \rho_*^T\left( \ev_1^*(\ch_T(\alpha_1)) \cdots 
  \ev_N^*(\ch_T(\alpha_N)) \cdot \Td^T(\Mb) \right)
\end{equation}
where $\Td^T(\Mb) \in \wh H_T(\Mb)$ is the equivariant Todd class of
the (singular) variety $\Mb_{0,N}(X,d)$.  Let $\Omega_1,\dots,\Omega_N
\subset X$ be $T$-stable closed subvarieties.  By using that
$\ch_T(y_i) = 1 - \exp(-y_i) = y_i + \text{higher terms} \in \wh
H_T(\pt)$ and $\ch_T(\cO_{\Omega_i}) = [\Omega_i] + $higher terms in
$\wh H_T(X)$ \cite[Thm.~18.3]{fulton:intersection}, we deduce from
(\ref{E:ekgw_lowest}) that the term of (lowest) degree $\sum
\codim(\Omega_i) - \dim \Mb_{0,N}(X,d)$ in the $K$-theoretic invariant
$I_d^T(\cO_{\Omega_1},\dots,\cO_{\Omega_N}) \in K^T(\pt) \subset
\Z\llbracket y_1,\dots,y_n \rrbracket$ is equal to the cohomological
invariant $I_d^T([\Omega_1],\dots,[\Omega_N])$.

We note that when $\sum \codim(\Omega_i) = \dim \Mb_{0,N}(X,d)$ we have
\[ I_d(\cO_{\Omega_1},\dots,\cO_{\Omega_N}) =
   I_d([\Omega_1],\dots,[\Omega_N]) = I_d^T([\Omega_1],\dots,[\Omega_N])
   \ \in \Z \,,
\]
but the equivariant $K$-theoretic invariant $I_d^T(\cO_{\Omega_1},
\dots, \cO_{\Omega_N}) \in K^T(\pt)$ may not be an integer.  In
general, the ordinary $K$-theoretic invariant
$I_d(\cO_{\Omega_1},\dots,\cO_{\Omega_N}) \in \Z$ is the total
dimension of the virtual representation $I_d^T(\cO_{\Omega_1}, \dots,
\cO_{\Omega_N}) \in K^T(\pt)$.

\begin{example}
  Let $X = \Gr(3,6)$ be the Grassmannian of 3-planes in $\C^6$, and
  let $T = (\C^*)^6$ act on $X$ through the coordinatewise action on
  $\C^6$.  Let $\Omega \subset X$ be the Schubert variety defined by
  the partition $\lambda = (2,1)$, i.e.\ $\Omega = \{ V \in X \mid V
  \cap \C^2 \neq 0 \text{ and } \dim(V \cap \C^4) \geq 2 \}$.  Then
  $I_0([\Omega],[\Omega],[\Omega]) = I^T_0([\Omega],[\Omega],[\Omega])
  = I_0(\cO_\Omega, \cO_\Omega, \cO_\Omega) = 2$, whereas
  $I_0^T(\cO_\Omega, \cO_\Omega, \cO_\Omega) = 1 + [\C_\chi]$ where
  the character $\chi$ is defined by $\chi(t_1,t_2,t_3,t_4,t_5,t_6) =
  (\frac{t_1 t_2}{t_5 t_6})^3$.
\end{example}

The non-equivariant cohomological (3-point) invariants of a
(generalized) flag manifold have been computed by exploiting the fact
that they are the structure constants for the quantum cohomology. We
will not attempt to survey the subject, but the reader can consult
\cite{bertram:quantum, ciocan-fontanine:on, fomin.gelfand.ea:quantum,
  fulton.woodward:on, woodward:on} and references therein.  The
equivariant cohomological invariants, which appear as structure
constants in equivariant quantum cohomology, have been computed in
\cite{mihalcea:equivariant*1, mihalcea:equivariant}, and more recent
algorithms can also be found in \cite{lam.shimozono:quantum}.

If the cohomology of the variety $X$ is generated by divisors (such as
$\P^r$ or a full flag manifold), then Lee and Pandharipande's
reconstruction theorem from \cite{lee.pandharipande:reconstruction}
can be used to compute the $N$-pointed $K$-theoretic invariants
starting from the 1-pointed invariants.  For projective spaces a
formula is known for the $J$-function, which encodes all the
$1$-pointed invariants \cite[\S
2.2]{lee.pandharipande:reconstruction}.  This yields a complete
algorithm to compute the K-theoretic invariants in this case.

We will proceed to express the (equivariant) $K$-theoretic
Gromov-Witten invariants of Grassmannians as triple intersections on
two-step flag manifolds, thus generalizing the identity proved in
\cite{buch.kresch.ea:gromov-witten}.

\subsection{Grassmannians of type A}\label{ssection:qkgrass}

Let $X = \Gr(m,n) = \{ V \subset \C^n : \dim V = m \}$ be the
Grassmann variety of $m$-planes in $\C^n$.  This variety has dimension
$mk$, where $k = n-m$; the dimension of the associated Kontsevich
moduli space $M_d := \Mb_{0,3}(X,d)$ is equal to $\dim X + nd$.
Following \cite{buch:quantum}, we define the {\em kernel\/} of a
3-pointed stable map $f : C \to X$ to be the intersection of the
$m$-planes $V \subset \C^n$ in its image, and we define the {\em
  span\/} of $f$ as the linear span of these subspaces.
\begin{equation}
  \Ker(f) = \bigcap_{V \in f(C)} V \ \ \ ; \ \ \ 
  \Span(f) = \sum_{V \in f(C)} V \ \subset \C^n \,.
\end{equation}
Given a degree $d \geq 0$ we set $a = \max(m-d,0)$ and $b =
\min(m+d,n)$.  If $f : C \to X$ is a stable map of degree $d$, then
its kernel and span satisfy the dimension bounds
\begin{equation} \label{E:dimbound}
  \dim \Ker(f) \geq a \ \ \text{ and } \ 
  \dim \Span(f) \leq b \,.
\end{equation}
This was proved in \cite[Lemma~1]{buch:quantum} when $C = \P^1$, and
in general it follows from this case by induction on the number of
components of $C$.  We will prove in Corollary~\ref{C:semicont} below
that $f \mapsto \dim \Ker(f)$ is an upper semicontinuous function on
$M_d$, while $f \mapsto \dim \Span(f)$ is lower semicontinuous.  Since
it is easy to construct stable maps $f : \bP^1 \to X$ for which the
bounds (\ref{E:dimbound}) are satisfied with equality (see
\cite[Prop.~1]{buch.kresch.ea:gromov-witten}), it follows that
(\ref{E:dimbound}) is satisfied with equality for all stable maps in a
dense open subset of $M_d$.

Define the two-step flag variety $Y_d = \Fl(a,b;n) = \{ A\subset B
\subset \C^n : \dim A = a, \dim B=b\}$ and the three-step flag variety
$Z_d = \Fl(a,m,b;n) = \{ A\subset V \subset B \subset \C^n : \dim A =
a, \dim V=m, \dim B=b\}$.  Let $p : Z_d \to X$ and $q : Z_d \to Y_d$
be the natural projections.  Our main result for Grassmannians of type
A is the following theorem.

\begin{thm} \label{T:mainA}
  For equivariant $K$-theory classes $\alpha_1, \alpha_2, \alpha_3 \in
  K^T(X)$ we have
\[ I^T_d(\alpha_1, \alpha_2, \alpha_3) = 
   \euler{Y_d}^T( q_*p^*(\alpha_1) \cdot q_*p^*(\alpha_2) 
   \cdot q_*p^*(\alpha_3) ) \,.
\]
\end{thm}

This generalizes \cite[Thm.~1]{buch.kresch.ea:gromov-witten} which
gives this identity for non-equivariant cohomological Gromov-Witten
invariants.  We note that the cohomological invariants vanish for
degrees $d$ larger than $\min(m,k)$, but this is not true for the
$K$-theoretic invariants.  The definition of $a$ and $b$ given above
is required to correctly compute the $K$-theoretic invariants for such
degrees.

Let $\Omega \subset X$ be a Schubert variety.  Like in \cite[\S
2]{buch.kresch.ea:gromov-witten} we define a modified Schubert variety
in $Y_d$ by
\[ \wt \Omega = q(p^{-1}(\Omega)) = 
   \{ (A,B) \in Y_d \mid \exists V \in \Omega : A \subset V \subset B \} \,.
\]
We then have $q_*p^*([\cO_\Omega]) = [\cO_{\wt \Omega}] \in
K_\circ(Y_d)$ by Lemma~\ref{L:pushschub}, and the cohomology class
$q_*p^*([\Omega])$ is equal to $[\wt \Omega] \in H^*(Y_d)$ if
$\codim(\wt \Omega) = \codim(\Omega) - d^2$ and is zero otherwise.  If
$\Omega$ is $T$-stable, then the same identities hold in equivariant
$K$-theory and cohomology.  We remark that the codimension of
$\wt\Omega$ equals $\codim(\Omega)-d^2$ if and only if the Young
diagram defining $\Omega$ contains a $d \times d$ rectangle.  In
particular we must have $d \leq \min(m,k)$.  The role of the $d \times
d$ rectangle was first discovered in \cite{yong:degree}, and the
geometric interpretation was given in
\cite{buch.kresch.ea:gromov-witten}.  We derive the following
corollary.

\begin{cor} \label{C:twostep}
  Let $\Omega_1, \Omega_2, \Omega_3 \subset X$ be Schubert varieties.
  Then we have
  \[ I_d(\cO_{\Omega_1}, \cO_{\Omega_2}, \cO_{\Omega_3}) =
  \euler{Y_d}([\cO_{\wt \Omega_1}] \cdot [\cO_{\wt \Omega_2}] \cdot
  [\cO_{\wt \Omega_3}])
  \] 
  and, if $d \leq \min(m,k)$, then
  \[ I_d([\Omega_1], [\Omega_2], [\Omega_3]) =
  \begin{cases}
    \int_{Y_d} [\wt\Omega_1] \cdot [\wt\Omega_2] \cdot [\wt\Omega_3] 
    & \text{if $\codim(\wt\Omega_i) = 
      \codim(\Omega_i)-d^2$ $\forall$ $i$,} \\
    0 & \text{otherwise.}
  \end{cases}
  \]
  If the $\Omega_1, \Omega_2, \Omega_3$ are $T$-stable, then these
  identities hold equivariantly as well.
\end{cor}

The second identity was proved in \cite{buch.kresch.ea:gromov-witten}
by putting the Schubert varieties in general position, and showing
that the map $f \mapsto (\Ker(f),\Span(f))$ gives a bijection between
the set of rational curves counted by the invariant $I_d([\Omega_1],
[\Omega_2], [\Omega_3])$ and the set of points in the intersection
$\wt \Omega_1 \cap \wt \Omega_2 \cap \wt \Omega_3$.  This approach
will not suffice for the general case, for example because the
equivariant and $K$-theoretic invariants do not have an enumerative
interpretation.  Instead, we will give a cohomological proof of
Theorem \ref{T:mainA}.  We need some notation.  For {\em arbitrary\/}
integers $a, b$ with $0 \leq a \leq m \leq b \leq n$ we define the
subset
\[ M_d(a,b) = \{ (A,B,f) \in \Fl(a,b;n) \times M_d \mid
   A \subset \Ker(f) \text{ and } \Span(f) \subset B \} \,.
\]

\begin{lemma}
  $M_d(a,b)$ is an irreducible closed subset of $\Fl(a,b;n) \times
  M_d$.  Furthermore, when equipped with the reduced scheme structure,
  $M_d(a,b)$ is a projective variety with at worst finite quotient
  singularities.
\end{lemma}
\begin{proof}
  Let $\cA \subset \cB \subset \C^n \times Y$ be the tautological flag
  on $Y = \Fl(a,b;n)$.  Then $M_d(a,b) = \{ (y,f) \in Y \times M_d
  \mid \cA(y) \subset \Ker(f) \text{ and } \Span(f) \subset \cB(y)
  \}$.  Let $U \subset Y$ be an open subset over which the
  tautological flag is isomorphic to the trivial flag $A_0 \times U
  \subset B_0 \times U \subset \C^n \times U$ given by some point
  $(A_0,B_0) \in Y$.  Let $\pr_1 : Y \times M_d \to Y$ be the first
  projection.  It follows from \cite[p.\ 
  12]{fulton.pandharipande:notes} that the subset of stable maps $f :
  C \to X$ in $M_d$ for which $f(C) \subset X' := \Gr(m-a,B_0/A_0)$ is
  closed and isomorphic to $\Mb_{0,3}(X',d)$ with its reduced
  structure.  Since the condition $f(C) \subset X'$ is equivalent to
  demanding that $A_0 \subset \Ker(f)$ and $\Span(f) \subset B_0$, we
  deduce that $M_d(a,b) \cap \pr_1^{-1}(U) = U \times \Mb_{0,3}(X',d)$
  is closed and reduced in $\pr_1^{-1}(U)$.  The lemma follows from
  this.
\end{proof}

\begin{cor} \label{C:semicont}
  The function $f \mapsto \dim \Ker(f)$ is upper semicontinuous on
  $M_d$, and $f \mapsto \dim \Span(f)$ is lower semicontinuous.
\end{cor}
\begin{proof}
  The set of stable maps $f$ in $M_d$ for which $\dim \Ker(f) \geq a$
  and $\dim \Span(f) \leq b$ is the image of the projective variety
  $M_d(a,b)$ under the projection $\Fl(a,b;n) \times M_d \to M_d$.
\end{proof}

From now on we set $a = \max(m-d,0)$ and $b = \min(m+d,n)$ as in the
statement of Theorem~\ref{T:mainA}.  In this case we write $\Bl_d =
M_d(a,b) = \{ (A,B,f) \in Y_d \times M_d \mid A \subset \Ker(f) \text{
  and } \Span(f) \subset B \}$.  We suspect that this variety is
isomorphic to the blowup of $M_d$ along the closed subset where the
kernel and span fail to have the expected dimensions $(a,b)$, but we
have not found a proof.  We also define the variety $Z^{(3)}_d =
\{(A,V_1,V_2,V_3,B) \mid (A,B) \in Y_d, V_i \in X, A \subset V_i
\subset B \}$.  Now construct the following commutative diagram, which
is the heart of the proof of Theorem~\ref{T:mainA}.
\begin{equation} \label{E:diagram}
  \xymatrix{\Bl_d \ar[rr]^{\pi} \ar[d]^{\phi} & & M_d \ar[d]^{\ev_i} \\ 
  Z_d^{(3)}\ar[r]^{e_i} & Z_d \ar[r]^{p} \ar[d]^{q} & X \\ & Y_d} 
\end{equation}
Here $\pi$ is the projection to the second factor of $Y_d \times M_d$,
and the other maps are given by $\phi(A,B,f) =
(A,\ev_1(f),\ev_2(f),\ev_3(f),B)$ and $e_i(A,V_1,V_2,V_3,B) =
(A,V_i,B)$.  All the maps in the diagram are $T$-equivariant.

\begin{lemma} \label{L:biratA}
  The map $\pi : \Bl_d \to M_d$ is birational.
\end{lemma}
\begin{proof}
  This map is surjective by the dimension bounds (\ref{E:dimbound}).
  Furthermore, for a dense open subset of stable maps $f$ in $M_d$ the
  dimension bounds (\ref{E:dimbound}) are satisfied with equality,
  which implies that $\pi^{-1}(f) = (\Ker(f),\Span(f),f)$.
\end{proof}

\begin{prop} \label{P:phi}
We have $\phi_* [\cO_{\Bl_d}] = [\cO_{Z_d^{(3)}}]$ in $K^T(Z_d^{(3)})$.
\end{prop}
\begin{proof}
  For any point $(A,B) \in Y_d$ we have $(q e_i \phi)^{-1}(A,B) =
  \Mb_{0,3}(X',d)$ where $X' = \Gr(m-a,B/A)$, and
  $\phi^{-1}(A,V_1,V_2,V_3,B)$ is the set of stable maps $f : C \to
  X'$ that send the three marked points to $V_1/A$, $V_2/A$, $V_3/A$.
  
  If $d \leq \min(m,k)$ then $X' = \Gr(d,2d)$.  Since the
  Gromov-Witten invariant $I_d(\pt,\pt,\pt)$ on $X'$ is equal to
  one, it follows that the general fiber of $\phi$ is a single point.
  This Gromov-Witten invariant can be computed with Bertram's
  structure theorems \cite{bertram:quantum} or by using
  \cite[Prop.~1]{buch.kresch.ea:gromov-witten}.  We conclude that
  $\phi$ is a birational isomorphism, and the proposition follows
  because both $\Bl_d$ and $Z_d^{(3)}$ have rational singularities.
  
  For arbitrary degrees $d$, the proposition follows from
  Theorem~\ref{T:push}, since the general fibers of $\phi$ are
  irreducible rational varieties by Corollary~\ref{C:ratA}.
\end{proof}

\begin{proof}[Proof of Theorem~\ref{T:mainA}]
  By using Lemma~\ref{L:biratA} and Proposition~\ref{P:phi}, it
  follows from the projection formula that
  \[ \begin{split}
   \euler{M_d}(\ev_1^*(\alpha_1) \cdot \ev_2^*(\alpha_2) \cdot
      \ev_3^*(\alpha_3)) 
   &= \euler{\Bl_d}(\pi^* \ev_1^*(\alpha_1) \cdot \pi^* \ev_2^*(\alpha_2)
      \cdot \pi^* \ev_3^*(\alpha_3)) \\
   &= \euler{Z_d^{(3)}}( e_1^*\, p^* (\alpha_1) \cdot 
      e_2^*\, p^* (\alpha_2) \cdot
      e_3^*\, p^* (\alpha_3) ) \,.
   \end{split}
  \]
  Since $Z_d^{(3)} = Z_d \times_{Y_d} Z_d \times_{Y_d} Z_d$, the rest
  follows from Lemma~\ref{L:product}.
\end{proof}

\begin{remark}
  The first author has conjectured a puzzle-based combinatorial
  formula for the structure constants of the equivariant cohomology of
  two-step flag varieties \cite[\S I.7]{coskun.vakil:geometric}, which
  generalizes results of Knutson and Tao \cite{knutson.tao:puzzles,
    knutson:conjectural}.  In view of Corollary~\ref{C:twostep}, this
  conjecture specializes to a Littlewood-Richardson rule for the
  equivariant quantum cohomology ring $\QH_T(X)$ of any Grassmannian.
  We refer to \cite[\S 2.4]{buch.kresch.ea:gromov-witten} for the
  translation.
\end{remark}



\section{Quantum $K$-theory of Grassmannians}
\label{S:qk}

\subsection{Definitions}

In this section we apply Theorem~\ref{T:mainA} to compute the quantum
$K$-theory of the Grassmannian $X = \Gr(m,n)$.  Recall that a
partition is a weakly decreasing sequence of non-negative integers
$\lambda = (\lambda_1 \geq \lambda_2 \geq \dots \geq \lambda_l \geq
0)$.  A partition can be identified with its Young diagram, which has
$\lambda_i$ boxes in row $i$.  If (the Young diagram of) $\lambda$ is
contained in another partition $\nu$, then $\nu/\lambda$ denotes the
skew diagram of boxes in $\nu$ which are not in $\lambda$.  The
Schubert varieties in $X$ are indexed by partitions contained in the
rectangular partition $(k)^m = (k,k,\dots,k)$ with $m$ rows and
$k=n-m$ columns.  The Schubert variety for $\lambda$ relative to the
Borel subgroup of $\GL(n)$ stabilizing a complete flag $0 = F_0
\subset F_1 \subset \dots \subset F_n = \C^n$ is defined by
\[ X_\lambda = \{ V \in X \mid \dim(V \cap F_{k+i-\lambda_i}) \geq i
   ~\forall 1 \leq i \leq m \} \,.
\]
The codimension of $X_\lambda$ in $X$ is equal to the weight
$|\lambda| = \sum \lambda_i$.  The Schubert classes $[X_\lambda]$
Poincar\'e dual to the Schubert varieties form a $\Z$-basis for the
cohomology ring $H^*(X)$.  The ordinary (small) quantum cohomology
ring of $X$ is an algebra over the polynomial ring $\Z[q]$, which as a
module is defined by $\QH(X) = H^*(X) \otimes_\Z \Z[q]$.  The ring
structure is given by
\begin{equation} \label{E:qcprd}
  [X_\lambda] \star [X_\mu] = \sum_{\nu,d \geq 0}
  I_d([X_\lambda],[X_\mu],[X_{\nu^\vee}])\, q^d\, [X_\nu]
\end{equation}
where the sum is over all partitions $\nu \subset (k)^m$ and
non-negative degrees $d$, and $\nu^\vee = (k-\nu_m, \dots, k-\nu_1)$
is the Poincar\'e dual partition of $\nu$.

The Grothendieck ring of $X$ has a basis consisting of the Schubert
structure sheaves $\cO_\lambda := [\cO_{X_\lambda}]$, $K(X) =
\bigoplus_\lambda \Z\cdot \cO_\lambda$.  The determinant of the
tautological subbundle on $X$ defines the class $t = 1 - \cO_{(1)}$ in
$K(X)$.  Define the $K$-theoretic dual Schubert class for $\lambda$ by
$\cO_\lambda^\vee = t\cdot \cO_{\lambda^\vee}$.  By
\cite[\S8]{buch:littlewood-richardson} we have the Poincar\'e duality
identity $\euler{X}(\cO_\lambda, \cO_\nu^\vee) =
\delta_{\lambda,\nu}$.

The quantum $K$-theory ring of $X$ is {\em not\/} obtained by
replacing the cohomological Gromov-Witten invariants with
$K$-theoretic invariants in (\ref{E:qcprd}), since this does not lead
to an associative ring. Instead we need the following definition of
structure constants, which comes from Givental's paper
\cite{givental:on}.  Given three partitions $\lambda, \mu, \nu$,
define the constant
\begin{equation}\label{E:defQKcoeff} 
  N^{\nu,d}_{\lambda,\mu} = \sum_{d_0,\dots,d_r,
  \kappa_1,\dots,\kappa_r} (-1)^r \, I_{d_0}(\cO_\lambda, \cO_\mu,
  \cO_{\kappa_1}^\vee) \left( \prod_{i=1}^{r-1}
  I_{d_i}(\cO_{\kappa_i},\cO_{\kappa_{i+1}}^\vee) \right) \,
  I_{d_r}(\cO_{\kappa_r}, \cO_\nu^\vee)
\end{equation}
where the sum is over all sequences of non-negative integers
$(d_0,\dots,d_r)$, $r \geq 0$, such that $\sum d_i = d$ and $d_i > 0$
for $i>0$, and all partitions $\kappa_1,\dots,\kappa_r$.  The
two-point invariants can be obtained using the identity
$I_d(\alpha_1,\alpha_2) = I_d(\alpha_1,\alpha_2,1)$, which holds
because the general fiber of the forgetful map $\Mb_{0,3}(X,d) \to
\Mb_{0,2}(X,d)$ is rational, see \cite[Cor.~1]{givental:on} or
Theorem~\ref{T:push}.

The constants $N^{\nu,d}_{\lambda,\mu}$ can also
be defined by the (equivalent) inductive identity 
\begin{equation}\label{E:QKrecurrence}
  N^{\nu,d}_{\lambda,\mu} = I_d(\cO_\lambda,\cO_\mu,\cO_\nu^\vee) -
  \sum_{\kappa,e>0} N^{\kappa,d-e}_{\lambda,\mu} \cdot
  I_e(\cO_\kappa,\cO_\nu^\vee) \,.
\end{equation} 
The degree zero constants $N^{\nu,0}_{\lambda,\mu}$ are the ordinary
$K$-theoretic Schubert structure constants, i.e.  $\cO_\lambda \cdot
\cO_\mu = \sum_\nu N^{\nu,0}_{\lambda,\mu} \, \cO_\nu$ in $K(X)$.

The $K$-theoretic quantum ring of $X$ is the $\Z\llbracket q
\rrbracket$-algebra given by $\QK(X) = K(X) \otimes_Z \Z\llbracket q
\rrbracket = \bigoplus_\lambda \Z\llbracket
q \rrbracket \, \cO_\lambda$ as a module, and the algebra structure is
defined by
\begin{equation} \label{E:qkgr}
  \cO_\lambda \star \cO_\mu = \sum_{\nu,d\geq 0} N^{\nu,d}_{\lambda,\mu}
  \, q^d \, \cO_\nu \,.
\end{equation}
It was proved by Givental that this product is associative
\cite{givental:on}.

\begin{remark}
In the definition of the ring $\QK(X)$ we have replaced the polynomial
ring $\Z[q]$ with the power series ring $\Z\llbracket q \rrbracket$,
since the structure constants $N^{\nu,d}_{\lambda,\mu}$ might be
non-zero for arbitrarily high degrees $d$.  In fact, the invariants
$I_d(\cO_\lambda, \cO_\mu, \cO_\nu)$ are equal to 1 for all
sufficiently large degrees $d$.  However, we will see in
Corollary~\ref{C:qkfinite} that $N^{\nu,d}_{\lambda,\mu}$ is zero when
$d > \ell(\lambda)$, so in fact we only work with polynomials in $q$.
It appears to be an open question if this also occurs for other
homogeneous spaces $G/P$.
\end{remark}

\begin{remark}
Similarly to ordinary $K$-theory, $\QK(X)$ admits a topological
filtration by ideals defined by $F_j \QK(X) = \bigoplus_{|\lambda|+ni
  \ge j} \Z \cdot q^i \cO_\lambda$, and the associated graded ring is the
ordinary quantum ring $\QH(X)$.
\end{remark}

\begin{remark} \label{R:eulerdiff}
  Implicit in Givental's proof that the $K$-theoretic quantum product
  is associative is the fact that the sum in (\ref{E:defQKcoeff}) can
  be interpreted as a difference between two Euler characteristics.
  More precisely, let $\mathcal{D}$ be the closure of the locus of
  maps in $\Mb_{0,3}(X,d)$ for which the domain has two components,
  the first and second marked points belong to one of these
  components, and the third marked point belongs to the other
  component.  This subvariety is a union of boundary divisors in
  $\Mb_{0,3}(X,d)$ (see e.g.\ \cite[\S 6]{fulton.pandharipande:notes})
  and these divisors have normal crossings, up to a finite group
  quotient (cf.\ Thm.~3 in {\em loc.~cit.}).  Then
  (\ref{E:defQKcoeff}) can be rewritten as:
  \begin{equation*} N_{\lambda,\mu}^{\nu,d} =
    \euler{\overline{\mathcal{M}}_{0,3}(X,d)}(\ev_1^*(\cO_\lambda)
    \cdot \ev_2^*(\cO_\mu) \cdot \ev_3^*(\cO_\nu^\vee)) -
    \euler{\mathcal{D}}(\ev_1^*(\cO_\lambda) \cdot \ev_2^*(\cO_\mu)
    \cdot \ev_3^*(\cO_\nu^\vee)) \,. 
  \end{equation*}
  This definition extends in an obvious manner to give the structure
  constants for the quantum $K$-theory of any homogeneous space $Y =
  G/P$.  It is interesting to ask if $\ev_* [\cO_{\Mb_{0,3}(Y,d)}] =
  [\cO_{Y \times Y \times Y}] = \ev_*[\cO_{\mathcal{D}}]$ for
  sufficiently large degrees $d$.  This would imply that $K$-theoretic
  quantum products are finite.
\end{remark}

\subsection{Pieri formula}

Our main result about the quantum $K$-theory of Grassmannians is a
Pieri formula for multiplying with the special classes $\cO_i =
\cO_{(i)}$ given by partitions with a single part.  It generalizes
Bertram's Pieri formula \cite{bertram:quantum} for the ordinary
quantum cohomology $H^*(X)$ as well as Lenart's Pieri formula in
ordinary $K$-theory \cite[Thm.~3.4]{lenart:combinatorial}.  Lenart's
formula states that $N^{0,\nu}_{i,\lambda}$ is non-zero only if
$\nu/\lambda$ is a horizontal strip, in which case
\begin{equation} \label{E:lenart}
  N^{\nu,0}_{i,\lambda} = (-1)^{|\nu/\lambda|-i} 
  \binom{r(\nu/\lambda)-1}{|\nu/\lambda|-i} \, 
\end{equation}
where $r(\nu/\lambda)$ is the number of non-empty rows in the skew
diagram $\nu/\lambda$.

Define the {\em outer rim\/} of the partition $\lambda$ to be the set
of boxes in its Young diagram that have no boxes strictly to the
South-East.  Any product of the form $\cO_i \star \cO_\lambda$ in $\QK(X)$
is determined by the following theorem combined with (\ref{E:lenart}).

\begin{thm} \label{T:qkpieri}
  The constants $N^{\nu,d}_{i,\lambda}$ are zero for $d \geq 2$.
  Furthermore, $N^{\nu,1}_{i,\lambda}$ is non-zero only if
  $\ell(\lambda)=m$ and $\nu$ can be obtained from $\lambda$ by
  removing a subset of the boxes in the outer rim of $\lambda$, with
  at least one box removed from each row.  When these conditions hold,
  we have
  \[ N^{\nu,1}_{i,\lambda} = (-1)^e \binom{r}{e} \]
  where $e = |\nu|+n-i-|\lambda|$ and $r$ is the number of rows of
  $\nu$ that contain at least one box from the outer rim of $\lambda$,
  excluding the bottom row of this rim.
\end{thm}

This result will be proved in the next section.

\begin{example}
  On $X = \Gr(3,6)$ we have $N^{(2,1),1}_{2,(3,2,1)} = -2$, since
  $e=1$ and $r=2$.  In fact, the partitions $\lambda = (3,2,1)$ and
  $\nu = (2,1)$ look as follows when the boxes inside the outer rim of
  $\lambda$ are shaded:
  \[ \lambda \ = \ \raisebox{-5mm}{\pic{.5}{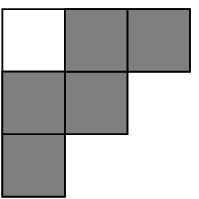}} 
     \hspace{10mm} \text{and} \hspace{10mm}
     \nu \ = \ \raisebox{-2mm}{\pic{.5}{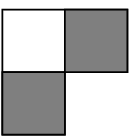}}
  \]
  This gives the negative coefficient of the product
  $\cO_2 \star \cO_{3,2,1} = \cO_{3,3,2} + q\,\cO_2 + q\,\cO_{1,1} -
  2\,q\,\cO_{2,1}$ in $\QK(X)$.
\end{example}

\subsection{Giambelli formula}

As a first application of the Pieri formula we derive a Giambelli
formula that expresses $K$-theoretic quantum Schubert classes as
polynomials in the special classes $\cO_i$, $1 \leq i \leq k$.  Let
$c(\nu/\lambda)$ denote the number of non-empty columns of the skew
diagram $\nu/\lambda$.  Given a partition $\mu$ of length $\ell$, we
let $\wh{\mu} = (\mu_1-1,\dots,\mu_\ell-1)$ be the partition obtained
by removing the first column from $\mu$.

\begin{thm}
  Let $a$ be an integer and $\mu$ a partition such that $\mu_1 \leq a
  \leq k$ and $0 < \ell(\mu) < m$.  In the quantum $K$-theory ring
  $\QK(X)$ we have
\[
  \cO_{a,\mu} = \sum_{p \geq a, \nu \subset \mu} (-1)^{|\mu/\nu|} 
  \binom{p-a-1+c(\nu/\wh\mu)}{p-a-|\mu/\nu|} \cO_p \star \cO_\nu \,,
\]
where the sum is over all integers $p \geq a$ and partitions $\nu$
contained in $\mu$ such that $\mu/\nu$ is a vertical strip.
\end{thm}
\begin{proof}
  Let $\groth_\lambda$ denote the stable Grothendieck polynomial for
  the partition $\lambda$, see \cite{buch:littlewood-richardson}.  The
  $K$-theoretic Jacobi-Trudi formula of
  \cite[Thm.~6.1]{buch:grothendieck} states that
\[ \groth_{a,\mu} = \groth_a \cdot \groth_\mu + \sum_{s \geq 1, t \geq 0}
   (-1)^s \binom{s-1+t}{t} \groth_{a+s+t}\cdot \groth_{\mu \sslash (1^s)} \,,
\]
where $\groth_{\mu\sslash(1^s)} = \sum_\nu d^\mu_{(1^s),\nu}
\groth_\nu$ is defined in terms of the coproduct coefficients of the
ring of stable Grothendieck polynomials.  According to
\cite[Cor.~7.1]{buch:littlewood-richardson} we have
\[ \groth_{\mu \sslash (1^s)} = \sum_\nu (-1)^{s-|\mu/\nu|}
   \binom{c(\nu/\wh\mu)}{s-|\mu/\nu|} \groth_\nu
\]
where the sum is over all partitions $\nu \subset \mu$ such that
$\mu/\nu$ is a vertical strip, and $c(\nu/\wh\mu)$ is the number of
non-empty columns in the skew diagram $\nu/\wh\mu$.

By using the identity $\sum_{t \geq 0} \binom{a}{t} \binom{b}{c-t} =
\binom{a+b}{c}$, these formulas combine to give the identity
\begin{equation} \label{E:Kgiam}
  \groth_{a,\mu} = \sum_{p \geq a, \nu \subset \mu} (-1)^{|\mu/\nu|} 
  \binom{p-a-1+c(\nu/\wh\mu)}{p-a-|\mu/\nu|} \groth_p \cdot \groth_\nu \,,
\end{equation}
where the sum is over all integers $p \geq a$ and partitions $\nu$
contained in $\mu$ such that $\mu/\nu$ is a vertical strip.

Since the stable Grothendieck polynomials represent $K$-theoretic
Schubert classes on Grassmannians, we may replace each stable
Grothendieck polynomial $\groth_\lambda$ in (\ref{E:Kgiam}) by the
corresponding class $\cO_\lambda$ in $K(X)$.  Finally, since the
partitions $\nu$ in (\ref{E:Kgiam}) satisfy $\ell(\nu) \leq \ell(\mu)
\leq m-1$, it follows from Theorem~\ref{T:qkpieri} that the ordinary
$K$-theory product $\cO_p \cdot \cO_\nu$ agrees with the quantum
product $\cO_p \star \cO_\nu$.  The theorem follows from this.
\end{proof}

\begin{cor}
  The class $\cO_\lambda \in \QK(X)$ can be expressed as a polynomial
  $P_\lambda = P_\lambda(\cO_1, \dots, \cO_k)$ in the special classes
  $\cO_i$.  The coefficients of this polynomial are integers, and each
  monomial involves at most $\ell(\lambda)$ special classes.
\end{cor}
\begin{proof}
  The polynomials $P_\lambda$ can be defined by induction on
  $\ell(\lambda)$ by setting $P_i = \cO_i$ for $1 \leq i \leq k$ and
  \[ P_\lambda = \sum_{p,\nu} (-1)^{|\mu/\nu|}
     \binom{p-\lambda_1-1+c(\nu/\wh\mu)}{p-\lambda_1-|\mu/\nu|}\, 
     \cO_p\, P_\nu
  \]
  when $\ell(\lambda) \geq 2$.  The sum is over $\lambda_1 \leq p \leq
  k$ and partitions $\nu$ contained in $\mu =
  (\lambda_2,\dots,\lambda_m)$ such that $\mu/\nu$ is a vertical
  strip.  Notice that $\ell(\nu) < \ell(\lambda)$, so the polynomials
  $P_\nu$ have already been defined.
\end{proof}

By using the polynomials $P_\lambda$, it becomes straightforward to
compute the structure constants $N^{\nu,d}_{\lambda,\mu}$ of $\QK(X)$.
In fact, we have $\cO_\lambda \star \cO_\mu = P_\lambda(\cO_1,\dots,\cO_k)
\star \cO_\mu$, and the latter product can be computed by letting
$P_\lambda$ act on $\cO_\mu$, with the action determined by the Pieri
formula of Theorem~\ref{T:qkpieri}.  Since multiplication by a single
special class can result in at most the first power of $q$, we deduce
that all exponents of $q$ in the product $\cO_\lambda \star \cO_\mu$ are
smaller than or equal to $\ell(\lambda)$.

\begin{cor} \label{C:qkfinite}
  The structure constant $N^{\nu,d}_{\lambda,\mu}$ is zero when $d >
  \ell(\lambda)$.
\end{cor}

The definition of the constants $N^{\nu,d}_{\lambda,\mu}$ in terms of
Gromov-Witten invariants can be turned around to give the identity
\begin{equation} \label{E:NtoGWd}
  I_d(\cO_\lambda, \cO_\mu, \cO_\nu^\vee) \ = \ 
  \sum_{\kappa, 0 \leq e \leq d}
  N^{\kappa,d-e}_{\lambda,\mu}\, I_e(\cO_\kappa,\cO_\nu^\vee) \,.
\end{equation}
Here we have used the convention that a two-point Gromov-Witten
invariant of degree zero is defined by the Poincar\'e pairing
$I_0(\alpha_1, \alpha_2) = \euler{X}(\alpha_1 \cdot \alpha_2)$.  By
linearity we similarly have that
\begin{equation} \label{E:NtoGWi}
  I_d(\cO_\lambda, \cO_\mu, \cO_\nu) \ = \ \sum_{\kappa,0\leq e\leq d}
  N^{\kappa,d-e}_{\lambda,\mu}\, I_e(\cO_\kappa,\cO_\nu) \,.
\end{equation}

Notice that all the required structure constants can be obtained by
computing the single product $\cO_\lambda \star \cO_\mu$ in $\QK(X)$.
Furthermore, it follows from Corollary~\ref{C:twopoint} below that
each two-point invariant $I_d(\cO_\kappa,\cO_\nu^\vee)$ is equal to
one if $\nu$ is obtained from $\kappa$ by removing its first $d$ rows
and columns, and is zero otherwise.  By using
\cite[Thm.~4.2.1]{brion:lectures} it also follows that
$I_d(\cO_\kappa,\cO_\nu)$ is equal to one if $\kappa_i + \nu_{m+d+1-i}
\leq k+d$ for $d < i \leq m$, and is zero otherwise.  The identities
(\ref{E:NtoGWd}) and (\ref{E:NtoGWi}) therefore give alternative and
very practical ways to compute $K$-theoretic Gromov-Witten invariants
on Grassmannians.

\begin{example}
  We compute the quantum product $\cO_{2,1} \star \cO_{2,1}$ on $X =
  \Gr(2,4)$.  The Giambelli formula gives $\cO_{2,1} = \cO_1 \star \cO_2$,
  so the product can be obtained as $\cO_{2,1} \star \cO_{2,1} = \cO_1 \star
  (\cO_2 \star \cO_{2,1}) = \cO_1 \star q\, \cO_1 = q\, \cO_{2} + q\,
  \cO_{1,1} - q\, \cO_{2,1} \in \QK(X)$.  Using (\ref{E:NtoGWd}) we
  obtain from this that $I_1(\cO_{2,1}, \cO_{2,1}, \cO_{2,1}^\vee) =
  N^{(2,1),1}_{(2,1),(2,1)} = -1$ and (\ref{E:NtoGWi}) gives
  $I_1(\cO_{2,1}, \cO_{2,1}, \cO_1) = 1 + 1 - 1 = 1$.  The full
  multiplication table for $\QK(\Gr(2,4))$ looks as follows.
\begin{align*}
\cO_{1} \star \cO_{1} &= \cO_{1,1}+\cO_{2}-\cO_{2,1} &
\cO_{1,1} \star \cO_{1} &= \cO_{2,1} \\
\cO_{1,1} \star \cO_{1,1} &= \cO_{2,2} &
\cO_{2} \star \cO_{1} &= \cO_{2,1} \\
\cO_{2} \star \cO_{1,1} &= q &
\cO_{2} \star \cO_{2} &= \cO_{2,2} \\
\cO_{2,1} \star \cO_{1} &= \cO_{2,2}+q -q\, \cO_{1} &
\cO_{2,1} \star \cO_{1,1} &= q\, \cO_{1} \\
\cO_{2,1} \star \cO_{2} &= q\, \cO_{1} &
\cO_{2,1} \star \cO_{2,1} &= q\, \cO_{1,1}+q\, \cO_{2}-q\, \cO_{2,1} \\
\cO_{2,2} \star \cO_{1} &= q\, \cO_{1} &
\cO_{2,2} \star \cO_{1,1} &= q\, \cO_{2} \\
\cO_{2,2} \star \cO_{2} &= q\, \cO_{1,1} &
\cO_{2,2} \star \cO_{2,1} &= q\, \cO_{2,1} \\
\cO_{2,2} \star \cO_{2,2} &= q^2
\end{align*}
\end{example}

We finally pose the following conjecture, which has been verified for
all Grassmannians $\Gr(m,n)$ with $n \leq 13$.

\begin{conj}
  The structure constants $N^{\nu,d}_{\lambda,\mu}$ have alternating
  signs in the sense that $\,(-1)^{|\nu|+nd-|\lambda|-|\mu|}\,
  N^{\nu,d}_{\lambda,\mu} \,\geq\, 0$.
\end{conj}

Lenart and Maeno have posed a similar conjecture for the quantum
$K$-theory of complete flag varieties $G/B$
\cite{lenart.maeno:quantum}.  Their conjecture is also based on
computer evidence, although these computations are based on other
conjectures.

We note that the Gromov-Witten invariants $I_d(\cO_\lambda, \cO_\mu,
\cO_\nu^\vee)$ do not have alternating signs for $d>0$, although the
degree zero invariants $I_0(\cO_\lambda, \cO_\mu, \cO_\nu^\vee) =
N^{\nu,0}_{\lambda,\mu}$ do have alternating signs
\cite{buch:littlewood-richardson}.  A concrete example on $\Gr(2,4)$
is $I_d(\cO_2, \cO_2, \cO_1^\vee) = I_d(\cO_2, \cO_{2,1},
  \cO_1^\vee) = 1$.  The invariants $I_d(\cO_\lambda, \cO_\mu,
  \cO_\nu)$ also do not have predictable signs even for $d=0$, see
\cite[\S 8]{buch:littlewood-richardson}.

\begin{example} \label{E:nonrat3pt}
  On $X = \Gr(4,8)$ we have $I_2(\cO_{4,3,2,1}, \cO_{4,3,2,1},
  \cO_{4,3,2,1}) = 2$.  The corresponding Gromov-Witten variety has
  dimension 2 and is not rational.
\end{example}

\subsection{Symmetry and duality}
\label{ssec:symdual}

As a further application of our Pieri rule, we prove some nice
properties of the ring $\QK(X)$.  Our first result shows that the
structure constants $N^{\nu^\vee,d}_{\lambda,\mu}$ satisfy
$S_3$-symmetry, i.e.\ these constants are invariant under arbitrary
permutations of the partitions $\lambda, \mu, \nu$.  We remark that
this symmetry is not at all clear from geometry.  For example, the two
terms in the expression for the structure constants in
Remark~\ref{R:eulerdiff} do not satisfy $S_3$-symmetry individually.

\begin{thm} \label{T:s3sym}
  For any degree $d$ and partitions $\lambda, \mu, \nu$ contained in
  the $m\times k$ rectangle we have $N^{\nu,d}_{\lambda,\mu} =
  N^{\mu^\vee,d}_{\lambda,\nu^\vee}$.
\end{thm}
\begin{proof}
  This identity is immediate from Theorem~\ref{T:qkpieri} if $\lambda =
  (p)$ has a single part.  In fact, $\nu$ is obtained by removing
  boxes from the outer rim of $\mu$ if and only if $\mu^\vee$ is
  obtained by removing boxes from the outer rim of $\nu^\vee$; and
  $\nu$ contains a box from the $i$-th row of the outer rim of $\mu$
  if and only if $\mu^\vee$ contains a box from the $(m-i)$-th row of
  the outer rim of $\nu^\vee$ ($1 \leq i \leq m-1$).
  
  For a partition $\nu$ and class $\alpha \in \QK(X)$, let $\langle
  \alpha, \cO_\nu \rangle \in \Z\llbracket q \rrbracket$ denote the
  coefficient of $\cO_\nu$ in the $\Z\llbracket q \rrbracket$-linear
  expansion of $\alpha$.  It is enough to show that the set
  \[ S = \{ \alpha \in \QK(X) \mid \langle \alpha \star \cO_\mu, \cO_\nu
  \rangle = \langle \alpha \star \cO_{\nu^\vee}, \cO_{\mu^\vee}
  \rangle \text{ for all partitions $\mu$ and $\nu$} \}
  \]
  is equal to $\QK(X)$.  This follows because $S$ is a $\Z\llbracket q
  \rrbracket$-submodule of $\QK(X)$ that contains the special classes
  $\cO_1, \dots, \cO_k$ and is closed under multiplication.  In fact,
  for $\alpha_1, \alpha_2 \in S$ we have $\langle \alpha_1 \star
  \alpha_2 \star \cO_\mu, \cO_\nu \rangle = \sum_\lambda \langle
  \alpha_1 \star \cO_\mu, \cO_\lambda \rangle \, \langle \alpha_2
  \star \cO_\lambda, \cO_\nu \rangle = \sum_\lambda \langle \alpha_2
  \star \cO_{\nu^\vee},\cO_{\lambda^\vee}\rangle\, \langle \alpha_1
  \star \cO_{\lambda^\vee},\cO_{\mu^\vee}\rangle= \langle \alpha_2
  \star \alpha_1 \star \cO_{\nu^\vee}, \cO_{\mu^\vee}\rangle$, so
  $\alpha_1 \star \alpha_2 \in S$.
\end{proof}

A skew diagram $\nu/\lambda$ is called a {\em rook strip\/} if each
row and column contains at most one box.  We need the following
special case of the Pieri rule.

\begin{lemma} \label{L:multbox}
  For any partition $\mu$ contained in the $m \times k$
  rectangle we have
  \[ \cO_1 \star \cO_\mu = \sum_\lambda (-1)^{|\lambda/\mu|}\, \cO_\lambda
  + q \sum_\nu (-1)^{|\nu/\wh{\wb{\mu}}|}\, \cO_\nu \,.
  \]
  Here $\wh{\wb{\mu}}$ denotes the result of removing the first row
  and first column from $\mu$.  The first sum is over all partitions
  $\lambda \supsetneq \mu$ for which $\lambda/\mu$ is a rook strip.
  The second sum is empty unless $\mu_1 = k$ and $\ell(\mu) = m$, in
  which case it includes all partitions $\nu \supset \wh{\wb{\mu}}$
  for which $\nu_1 = k-1$, $\ell(\nu) = m-1$, and $\nu/\wh{\wb{\mu}}$
  is a rook strip.
\end{lemma}

Let $\chi^q : \QK(X) \to \Z\llbracket q \rrbracket$ denote the
$\Z\llbracket q \rrbracket$-linear extension of the Euler
characteristic, which maps each Schubert class $\cO_\lambda$ to $1$.
Define the element $t_q = \frac{1-\cO_1}{1-q} \in \QK(X)$.  We finally
show that $t_q \star \cO_{\lambda^\vee}$ is the $\Z\llbracket q
\rrbracket$-linear dual basis element of $\cO_\lambda$.

\begin{thm} \label{T:qkdual}
  For any partitions $\lambda$ and $\nu$ contained in the $m\times k$
  rectangle we have $\chi^q(\cO_\lambda \star t_q \star
  \cO_{\nu^\vee}) = \delta_{\lambda,\nu}$.
\end{thm}
\begin{proof}
  Let $R = (k)^m$ denote the $m \times k$ rectangle considered as a
  partition.  It follows from Lemma~\ref{L:multbox} that $\chi^q((1-\cO_1)
  \star \cO_\lambda) = (1-q)\, \delta_{\lambda,R}$.  We deduce that
\[\begin{split}
  \chi^q((1-\cO_1)\star \cO_\lambda \star \cO_{\nu^\vee}) 
  &= \sum_{\mu,d} N^{\mu,d}_{\lambda,\nu^\vee}\, q^d\, 
     \chi^q((1-\cO_1) \star \cO_\mu)
  = \sum_d N^{R,d}_{\lambda,\nu^\vee}\, q^d\, (1-q) \\
  &= \sum_d N^{\nu,d}_{\lambda,\emptyset}\, q^d\, (1-q) 
  = \delta_{\lambda,\nu}\, (1-q)
  \,,
\end{split}\]
as required.  The third equality follows from Theorem~\ref{T:s3sym}.
\end{proof}

It follows from Theorem~\ref{T:qkdual} that the structure constants of
$\QK(X)$ can be expressed in the following form, which makes the
$S_3$-symmetry apparent:

\begin{equation} \label{E:sym3euler}
  \sum_{d \geq 0} N^{\nu^\vee,d}_{\lambda,\mu}\, q^d \ = \ 
  \chi^q(t_q \star \cO_\lambda \star \cO_\mu \star \cO_\nu) \,.
\end{equation}

\begin{example}
  On $X = \Gr(5,10)$ we have $\chi^q(t_q \star
  \cO_{(5,4,3,2,1)}^{\,3}) = 14\, q^2 + q^3$, so the sum
  (\ref{E:sym3euler}) may have more than one non-zero term.  However,
  the sum is always finite by Corollary~\ref{C:qkfinite}.
\end{example}

In the special case of ordinary $K$-theory, the $S_3$-symmetry of the
structure constants $N^{\nu^\vee,0}_{\lambda,\mu}$ follows from a
Puzzle version of the Littlewood-Richardson rule for $K(X)$
\cite[Thm.~4.6]{vakil:geometric}.  The $q=0$ case of
(\ref{E:sym3euler}) was proved in \cite[\S 2]{buch:combinatorial},
which provides an alternative proof.  
\ignore{
  If $X = G/P$ is any homogeneous space defined by a maximal parabolic
  subgroup $P$, then the basis of $K(X)$ consisting of Schubert
  structure sheaves can be dualized by multiplying all classes by a
  constant element (the ideal sheaf of the unique Schubert divisor).
  Theorem~\ref{T:qkdual} shows that this phenomenon carries over to
  quantum $K$-theory for Grassmannians of type A.
}
Theorem~\ref{T:qkdual} shows that the phenomenon that the Schubert
basis of $K(X)$ can be dualized by multiplying all structure sheaves
with a constant element carries over to quantum $K$-theory.
It would be interesting to find an explanation in terms of generating
functions.  The $S_3$-symmetry of Theorem~\ref{T:s3sym} might hint
towards the existence of a puzzle rule for the structure constants
$N^{\nu,d}_{\lambda,\mu}$ of $\QK(X)$, but so far we have not been
able to find a working set of puzzle pieces.

\subsection{Equivariant quantum $K$-theory of $\P^1$ and $\P^2$}
\label{ssec:qktproj}

Let $T \subset \GL_n$ be the torus of diagonal matrices, and let
$\cO_\lambda \in K_T(X)$ denote the equivariant class of the Schubert
variety $X_\lambda$ in $X = \Gr(m,n)$ relative to the standard
$T$-stable flag $F_\bull$ defined by $F_i = \C^i \oplus 0^{n-i}$.  The
$T$-equivariant quantum ring $\QK_T(X)$ is obtained by using
equivariant Gromov-Witten invariants in the definition
(\ref{E:QKrecurrence}) of the structure constants
$N^{\nu,d}_{\lambda,\mu}$, where $\cO_\lambda^\vee$ denotes the
equivariant Poincar\'e dual class of $\cO_\lambda$.  We include here
the multiplication tables for the equivariant quantum $K$-theory rings
of $\P^1 = \Gr(1,2)$ and $\P^2 = \Gr(1,3)$.  By Theorem~\ref{T:mainA},
the required Gromov-Witten invariants can be computed in $K_T(\pt)$,
$K_T(\P^1)$, and $K_T(\P^2)$ (see also Theorem~\ref{T:specialgw}
below).  We note that Graham and Kumar have given explicit formulas
for multiplication in $K_T(\P^n)$ \cite[\S 6.3]{graham.kumar:on}.  Let
$\varepsilon_i : T \to \C^*$ be the character defined by
$\varepsilon_i(t_1,\dots,t_n) = t_i$, and write $e^{\varepsilon_i} =
[\C_{\varepsilon_i}] \in K_T(\pt)$.

The multiplicative structure of $\QK_T(\P^1)$ is determined by:
\[ \cO_1 \star \cO_1 \ = \ (1 - e^{\varepsilon_1 - \varepsilon_2})\,
   \cO_1 + e^{\varepsilon_1 - \varepsilon_2}\, q
\]
And the multiplicative structure of $\QK_T(\P^2)$ is determined by:
\[\begin{split}
  \cO_{1}\star \cO_{1} \ &= \ (1-e^{\varepsilon_2-\varepsilon_3})\, \cO_1
    +e^{\varepsilon_2-\varepsilon_3}\, \cO_{2} \\
  \cO_{1}\star \cO_{2} \ &= \ e^{\varepsilon_1-\varepsilon_3}\, q
    +(1-e^{\varepsilon_1-\varepsilon_3})\, \cO_2 \\
  \cO_2\star \cO_2 \ &= \ (1-e^{\varepsilon_1-\varepsilon_2})\,
    (1-e^{\varepsilon_1-\varepsilon_3})\, \cO_2  
    + e^{\varepsilon_1-\varepsilon_3}(1-e^{\varepsilon_1-\varepsilon_2})\,q  
    + e^{\varepsilon_1-\varepsilon_2}\, q\, \cO_1 
\end{split}\]

Notice that the structure constants $N^{\nu,d}_{\lambda,\mu}$
appearing in these examples satisfy Griffeth-Ram positivity
\cite{griffeth.ram:affine} in the sense that
\[ (-1)^{|\nu|+nd-|\lambda|-|\mu|}\, N^{\nu,d}_{\lambda,\mu} \ \in \ 
   \N[e^{\varepsilon_1-\varepsilon_2}-1, \dots, 
   e^{\varepsilon_{n-1}-\varepsilon_n}-1] \,.
\]
Using Theorem~\ref{T:mainA}, we have verified that this positivity
also holds for all Grassmannians $\Gr(m,n)$ with $n \leq 5$, and it is
natural to conjecture that it holds in general.  In the case of
ordinary equivariant $K$-theory, this has been proved in
\cite{anderson.griffeth.ea:positivity}.



\section{Proof of the Pieri formula}
\label{S:proof}

\subsection{Special Gromov-Witten invariants}

To prove the Pieri formula, we start by establishing a formula for
certain special Gromov-Witten invariants on Grassmannians.  Fix a
degree $d \geq 0$.  As usual we set $a = \max(m-d,0)$, $b =
\min(a+d,n)$, $Y_d = \Fl(a,b;n)$, and $Z_d = \Fl(a,m,b;n)$.  We also
define $X_d = \Gr(b,n)$.  The following commutative diagram was
exploited earlier to obtain a quantum Pieri formula for submaximal
orthogonal Grassmannians \cite{buch.kresch.ea:quantum}.  It was also
applied to Grassmannians of type A in \cite{tamvakis:quantum}.  All
maps in the diagram are projections.

\[\xymatrix{
  Z_d \ar[r]^{p_1\hmm{8}} \ar[d]^{q} & 
  \Fl(m,b;n) \ar[r]^{\hmm{8}p_2} \ar[d]^{q'} & X \\ 
  Y_d \ar[r]^{\pi} & X_d}
\]

Given a partition $\lambda$, let $\wh{\lambda}$ denote the partition
obtained by removing the first $d$ columns, i.e.\ $\wh{\lambda}_i =
\max(\lambda_i-d,0)$.  If $X_\lambda$ is a Schubert variety in $X$,
then $q'(p_2^{-1}(X_\lambda))$ is the Schubert variety in $X_d$
defined by $\wh\lambda$.  It follows from Lemma~\ref{L:pushschub} that
$q'_* p_2^*(\cO_\lambda) = \cO_{\wh\lambda} \in K(X_d)$.  Similarly,
if $\lambda$ is a partition contained in the $b \times (n-b)$
rectangle, then ${p_2}_* {q'}^*(\cO_\lambda) = \cO_{\wb\lambda}$,
where $\wb\lambda = (\lambda_{d+1}, \dots, \lambda_b)$ is the
partition obtained by removing the top $d$ rows of $\lambda$.  We will
occasionally write $\wh\lambda(d) = \wh\lambda$ and $\wb\lambda(d) =
\wb\lambda$ to avoid ambiguity about how many rows or columns to
remove.

\begin{thm} \label{T:specialgw}
  Let $\lambda$ be a partition with $\ell(\lambda) \leq d$.  For any
  classes $\alpha_1, \alpha_2 \in K(X)$ we have
\[ I_d(\cO_\lambda, \alpha_1, \alpha_2) = \euler{X_d}(\cO_{\wh\lambda(d)}
   \cdot q'_* p_2^*(\alpha_1) \cdot q'_* p_2^*(\alpha_2)) \,.
\]
\end{thm}

\begin{proof}
  If $\Omega \subset \Fl(m,b;n)$ is any subset, then
  $p_1(q^{-1}(q(p_1^{-1}(\Omega)))) = q'^{-1}(q'(\Omega))$.  By taking
  $\Omega$ to be a Schubert variety, it follows from
  Lemma~\ref{L:pushschub} that ${p_1}_* q^* q_* p_1^* [\cO_\Omega] =
  {q'}^* q'_* [\cO_\Omega]$.  We deduce that for arbitrary classes
  $\beta_1, \beta_2 \in K(\Fl(m,b;n))$ we have
\[\begin{split} 
  & \pi_*(q_*p_1^*(\beta_1) \cdot q_*p_1^*(\beta_2)) 
  = \pi_*q_*(q^*q_*p_1^*(\beta_1) \cdot p_1^*(\beta_2))
  = q'_*{p_1}_*(q^*q_*p_1^*(\beta_1) \cdot p_1^*(\beta_2)) \\
  &= q'_*({p_1}_*q^*q_*p_1^*(\beta_1) \cdot \beta_2)
  = q'_*({q'}^*q'_*(\beta_1) \cdot \beta_2)
  = q'_*(\beta_1) \cdot q'_*(\beta_2) \ \in K(X_d) \,.
\end{split}\]

Since $\ell(\lambda) \leq d$, it follows by checking Schubert
conditions that $q(p^{-1}(X_\lambda)) = \pi^{-1}(X_{\wh\lambda})$,
where $p = p_2 p_1$ and the Schubert variety $X_{\wh\lambda}$ lives in
the Grassmannian $X_d$.  This implies that $q_*p^*(\cO_\lambda) =
\pi^*(\cO_{\wh\lambda})$.  We deduce from Theorem~\ref{T:mainA} that
\[\begin{split}
  I_d(\cO_\lambda,\alpha_1,\alpha_2) 
  &= \euler{Y_d}( q_*p^*(\cO_\lambda) \cdot q_*p^*(\alpha_1) \cdot
      q_*p^*(\alpha_2) ) \\
  &= \euler{Y_d}( \pi^*(\cO_{\wh\lambda}) \cdot q_*p^*(\alpha_1) \cdot
     q_*p^*(\alpha_2) ) \\
  &= \euler{X_d}( \cO_{\wh\lambda} \cdot q'_*p_2^*(\alpha_1) \cdot 
     q'_*p_2^*(\alpha_2) )
\end{split}\]
as required.
\end{proof}

Let $\wh{\wb\lambda}(d)$ denote the result of removing the first $d$
rows and the first $d$ columns from $\lambda$.

\begin{cor} \label{C:twopoint}
  For any partition $\mu$ contained in the $m \times k$ rectangle and
  any class $\alpha \in K(X)$ we have $I_d(\cO_\mu, \alpha) =
  \euler{X}(\cO_{\wh{\wb\mu}(d)}, \alpha)$.
\end{cor}
\begin{proof}
  By taking $\lambda$ to be the empty partition, we obtain
\[\begin{split} 
  I_d(\cO_\mu, \alpha)
   &= \euler{X_d}(q'_* p_2^*(\cO_\mu) \cdot q'_* p_2^*(\alpha))
   = \euler{X_d}(\cO_{\wh\mu} \cdot q'_* p_2^*(\alpha)) \\
   &= \euler{\Fl(m,b;n)}( q'^*(\cO_{\wh\mu}) \cdot p_2^*(\alpha) )
   =  \euler{X}({p_2}_* {q'}^*(\cO_{\wh\mu}) \cdot \alpha) 
   =  \euler{X}(\cO_{\wh{\wb\mu}} \cdot \alpha)
\end{split}\]
  as claimed.
\end{proof}

We note that Theorem~\ref{T:specialgw} and Corollary~\ref{C:twopoint}
have straightforward generalizations to $K$-equivariant Gromov-Witten
invariants, with the same proofs.

\subsection{Pieri coefficients}

In the following we set $\cO_i = 1 \in K(X)$ for $i \leq 0$.  For the
statement of the next lemma we need to remark that the structure
constants $N^{\nu,0}_{\lambda,\mu}$ of degree zero are independent of
the Grassmannian on which they are defined: they appear in the
multiplication of stable Grothendieck polynomials $\groth_\lambda
\cdot \groth_\mu = \sum N_{\lambda, \mu}^{\nu, 0} \groth_\nu$.  In
particular, these coefficients are well defined when the partitions
$\lambda$, $\mu$, and $\nu$ are not contained in the $m \times k$
rectangle.

\begin{lemma} \label{L:pushdiff}
  Let $\lambda$ be a partition contained in the $m\times k$ rectangle
  and $0 \leq i \leq m$.  Then we have
  \[ \cO_{i-d} \cdot \cO_{\wh\lambda(d)} 
     - q'_* p_2^*(\cO_i \cdot \cO_\lambda) 
     = \sum_{\ell(\mu)=m+1} N^{\mu,0}_{i,\lambda}\, \cO_{\wh\mu(d)}
  \]
  in $K(X_d)$, where the sum is over all partitions $\mu$ with exactly
  $m+1$ rows and at most $k$ columns.
\end{lemma}
\begin{proof}
  By using that $\Fl(m,b;n)$ is a Grassmann bundle over $X_d$, this is
  a special case of \cite[Cor.~7.4]{buch:grothendieck}, cf.\ 
  \cite[\S8]{buch:littlewood-richardson}.  Notice that Lenart's Pieri
  rule (\ref{E:lenart}) implies that $N^{\mu,0}_{i,\lambda}$ is zero
  whenever $\ell(\mu) \geq m+2$.
\end{proof}

\begin{cor}\label{cor:diff}
  Let $\lambda$ be contained in the $m \times k$ rectangle, $0\leq
  i\leq m$, and $\alpha \in K(X)$.  Then we have
  \[ I_d(\cO_i, \cO_\lambda, \alpha) - I_d(\cO_i \cdot \cO_\lambda,
     \alpha) = 
     \sum_{\ell(\mu)=m+1} N^{\mu,0}_{i,\lambda}\,
     \euler{X}(\cO_{\wh{\wb\mu}(d)} \cdot \alpha) \,.
  \]
\end{cor}
\begin{proof}
  By Theorem~\ref{T:specialgw} and Lemma~\ref{L:pushdiff} we have
  \[\begin{split}
  & I_d(\cO_i, \cO_\lambda, \alpha) - I_d(\cO_i \cdot \cO_\lambda,\alpha) \\
  &= \euler{X_d}(\cO_{i-d} \cdot \cO_{\wh\lambda} \cdot q'_*
  p_2^*(\alpha)) - \euler{X_d}(q'_* p_2^*(\cO_i \cdot \cO_\lambda)
  \cdot q'_* p_2^*(\alpha)) \\
  &= \sum_{\ell(\mu)=m+1} N^{\mu,0}_{i,\lambda}\,
  \euler{X_d}(\cO_{\wh\mu(d)} \cdot q'_* p_2^*(\alpha) ) \,.
  \end{split}\]
  Finally, the projection formula implies that
  $\euler{X_d}(\cO_{\wh\mu(d)} \cdot q'_* p_2^*(\alpha)) =
  \euler{X}(\cO_{\wh{\wb\mu}(d)} \cdot \alpha)$.
\end{proof}

\begin{proof}[Proof of Theorem~\ref{T:qkpieri}]
  The Pieri coefficients of degree one are given by
\begin{equation}\label{E:deg1signs}
\begin{split}
  N^{\nu,1}_{i,\lambda}
  &= I_1(\cO_i, \cO_\lambda, \cO_\nu^\vee) - \sum_\kappa
     N^{\kappa,0}_{i,\lambda}\, I_1(\cO_\kappa, \cO_\nu^\vee) \\
  &= I_1(\cO_i, \cO_\lambda, \cO_\nu^\vee) -
     I_1(\cO_i \cdot \cO_\lambda, \cO_\nu^\vee) \\
  &= \sum_{\ell(\mu)=m+1} N^{\mu,0}_{i,\lambda}\,
     \euler{X}(\cO_{\wh{\wb\mu}(1)} \cdot \cO_\nu^\vee)
   = \sum_{j=\nu_1+1}^k N^{(j,\nu+1^m),0}_{i,\lambda} \,.
\end{split}
\end{equation}
Here we used Cor.~\ref{cor:diff} and the Poincar\'e duality.  Notice
that this implies that
\[ \sum_\nu N^{\nu,1}_{i,\lambda} \cO_\nu = 
   \sum_{\ell(\mu)=m+1} N^{\mu,0}_{i,\lambda}\, \cO_{\wh{\wb\mu}(1)}
   \,.
\]
It follows that for any degree $d \geq 0$ we have
\[\begin{split}
   I_d(\cO_i, \cO_\lambda, \cO_\nu^\vee) 
     - I_d(\cO_i \cdot \cO_\lambda, \cO_\nu^\vee) 
   &= \sum_{\ell(\mu)=m+1} N^{\mu,0}_{i,\lambda}\,
      \euler{X}(\cO_{\wh{\wb\mu}(d)} \cdot \cO_\nu^\vee) \\
   &= \sum_{\kappa} N^{\kappa,1}_{i,\lambda}\,
      \euler{X}(\cO_{\wh{\wb\kappa}(d-1)}, \cO_\nu^\vee) \\
   &= \sum_\kappa N^{\kappa,1}_{i,\lambda}\, I_{d-1}(\cO_\kappa,
      \cO_\nu^\vee) \,.
\end{split}\]
This identity implies that $N^{\nu,d}_{i,\lambda} = 0$ for $d\geq 2$
by induction on $d$.

We finally prove that the constants $N^{\nu,1}_{i,\lambda}$ of degree
one are given by the signed binomial coefficients of
Theorem~\ref{T:qkpieri}.  Define a {\em marked horizontal strip\/} for
the pair $(i,\lambda)$ to be a horizontal strip $D$ of some shape
$\mu/\lambda$, for which $|D|-i$ of the non-empty rows of $D$ are
marked, excluding the bottom row.  Then Lenart's Pieri formula
(\ref{E:lenart}) states that
\[ N^{\mu,0}_{i,\lambda} = \sum_{D:\,\sh(D)=\mu/\lambda} (-1)^{|D|-i} \]
and we obtain from (\ref{E:deg1signs}) that
\begin{equation}\label{E:beforeinv} 
  N^{\nu,1}_{i,\lambda} = 
  \sum_D (-1)^{|D|-i}
\end{equation}
where this sum is over all marked horizontal strips $D$ for
$(i,\lambda)$ of some shape $(j,\nu+1^m)/\lambda$ with $\nu_1 + 1 \leq
j \leq k$.  

It turns out that the sum (\ref{E:beforeinv}) does not change if we
include only marked horizontal strips $D$ of shape
$(k,\nu+1^m)/\lambda$ such that the top row of $D$ is not marked.
This follows from the sign reversing involution that sends any $D =
(j,\nu+1^m)/\lambda$ with $j < k$ and the top row unmarked to $D' =
(j+1,\nu+1^m)/\lambda$ with the top row marked; and which sends $D' =
(j,\nu+1^m)/\lambda$ with the top row marked to $D =
(j-1,\nu+1^m)/\lambda$ with the top row unmarked.  Note that if the
top row of $(j,\nu+1^m)$ is marked, then this row is not empty and $j
> \nu_1+1$.

We finally notice that $(k,\nu+1^m)/\lambda$ is a horizontal strip
if and only if $\nu$ is obtained from $\lambda$ by removing a subset
of the boxes in the outer rim of $\lambda$, with at least one box
removed in each of the $m$ rows.  And the $i$-th row of
$(k,\nu+1^m)/\lambda$ is non-empty if and only if the $(i-1)$-st row
of $\nu$ contains a box from the outer rim of $\lambda$.
\end{proof}


\section{Gromov-Witten invariants of cominuscule varieties}
\label{S:comin}

In this last section we generalize our formula for Grassmannian
Gromov-Witten invariants to work for all cominuscule homogeneous
spaces, with the exception that $K$-theoretic invariants can be
computed for ``{\em small\/}'' degrees only.  The computation of
Gromov-Witten invariants in Theorem~\ref{T:mainA} extends almost
verbatim to Lagrangian and maximal orthogonal Grassmannians by using a
case by case analysis as in \cite{buch.kresch.ea:gromov-witten}.
However, we will utilize here the unified approach of Chaput, Manivel,
and Perrin \cite{chaput.manivel.ea:quantum}, which makes it possible
to state and prove our result in a type independent manner.

A cominuscule variety is a homogeneous space $X = G/P$, where $G$ is a
simple complex linear algebraic group and $P \subset G$ is a parabolic
subgroup corresponding to a cominuscule simple root $\alpha$.  The
latter means that when the highest root is expressed as a linear
combination of simple roots, the coefficient of $\alpha$ is one.
Since $P$ is maximal we have $H_2(X) \cong \Z$, so the degree of a
stable map to $X$ can be identified with a non-negative integer.  The
family of cominuscule varieties include Grassmannians of type A,
Lagrangian Grassmannians $\LG(m,2m)$, maximal orthogonal Grassmannians
$\OG(m,2m)$, quadric hypersurfaces $\Q^n \subset \P^{n+1}$, as well as
two exceptional varieties called the Cayley plane and the Freudenthal
variety.  All minuscule varieties are also cominuscule.  We refer to
\cite{buch.kresch.ea:gromov-witten, chaput.manivel.ea:quantum} for
more details.


Given two points $x,y \in X$, we let $d(x,y)$ denote the smallest
possible degree of a stable map $f : C \to X$ with $x,y \in f(C)$.
This definition of F.~L.~Zak \cite{zak:tangents} gives $X$ the
structure of a metric space.  Let $X(d)_{x,y}$ denote the union
of the images of all such stable maps $f$ of degree $d = d(x,y)$:
\[ X(d)_{x,y} \ = \ \bigcup_{\deg(f)=d \,;\, 
   x,y\in f(C)} f(C) \ \subset \ X \,. 
\]
It was proved in \cite{chaput.manivel.ea:quantum} that this set is a
Schubert variety in $X$.  Write $X(d)$ for the abstract variety
defined by $X(d)_{x,y}$.  Let $Y_d$ be a set parametrizing all
varieties $X(d)_{x,y}$ for $x,y \in X$ with $d(x,y)=d$, and let $G$
act on this set by translation.  For an element $\omega \in Y_d$ we
let $X_\omega \subset X$ denote the corresponding variety.  We also
set $d_{\max} = d_{\max}(X) = \max \{ d(x,y) \mid x,y \in X \}$.  A
root theoretic interpretation of this number can be found in
\cite[Def.~3.15]{chaput.manivel.ea:quantum}.  A degree $d$ is {\em
  small\/} if $d \leq d_{\max}$.  The following statement combines
Prop.~3.16, Prop.~3.17, and Fact~3.18 of
\cite{chaput.manivel.ea:quantum}.

\begin{prop}[Chaput, Manivel, Perrin] \label{P:cmp}
Let $d \leq d_{\max}$ be a small degree.
\begin{abcenum}
\item 
  The metric $d(x,y)$ attains all values between $0$ and $d_{\max}$.

\item
  $G$ acts transitively on the set of pairs $(x,y) \in X \times
  X$ with $d(x,y)=d$.
  
\item 
  Let $\omega \in Y_d$.  The stabilizer $G_\omega \subset G$ of
  $X_\omega$ is a parabolic subgroup of $G$ that acts transitively on
  $X_\omega$.
   
\item 
  Given a stable map $f : C \to X$ of degree $d$, there exists a
  point $\omega \in Y_d$ such that $f(C) \subset X_\omega$.  If $f \in
  \Mb_{0,3}(X,d)$ is a general point, then $\omega$ is uniquely
  determined.
  
\item 
  Let $\omega \in Y_d$ and let $x,y,z \in X_\omega$ be three general
  points.  Then there exists a unique stable map $f : C \to X$ of
  degree $d$ that sends the three marked points to $x$, $y$, and $z$.
  Furthermore we have $f(C) \subset X_\omega$.
\end{abcenum}
\end{prop}

Parts (b) and (c) of this proposition imply that $Y_d$ is a
homogeneous $G$-variety.  The idea of Chaput, Manivel, and Perrin's
construction is that the kernel-span pairs known from the classical
types are replaced by points in the variety $Y_d$, and the condition
that an $m$-plane in $\Gr(m,n)$ lies between a given kernel-span pair
$\omega$ is replaced with the condition that a point of $X$ belongs to
$X_\omega$.  For the cominuscule Grassmannians of types A, C, and D,
the varieties $Y_d$ and $X(d)$ are given in the following table; a
complete list can be found in
\cite[Prop.~3.16]{chaput.manivel.ea:quantum}.

\[\begin{array}{ccccc}\label{table:X(d)}
 X & d_{\max} & Y_d & X(d) \\
\hline
\Gr(m,n) & \min(m,n-m) & \Fl(m-d,m+d;n) & \Gr(d,2d) \\
\LG(n,2n) & n & \IG(n-d,2n) & \LG(d,2d) \\
\OG(n,2n) & \lfloor\frac{n}{2}\rfloor & \OG(n-2d,2n) &  \OG(2d,4d) \\ 
\end{array}
\vspace{2mm}
\]

Define the incidence variety $Z_d = \{ (\omega,x) \in Y_d \times X
\mid x \in X_\omega \}$.  It follows from part (c) of the proposition
that the diagonal action of $G$ on this set is transitive.
Furthermore, since $X_\omega$ is a Schubert variety and all Borel
subgroups in $G$ are conjugate, one can choose $\omega \in Y_d$ such
that the $G_\omega$ and $P$ both contain a common Borel subgroup of
$G$.  Since this Borel subgroup will be contained in the stabilizer of
the point $(\omega_0, P/P) \in Z_d$, it follows that $Z_d$ is also a
homogeneous space for $G$.  We also need the varieties $M_d =
\Mb_{0,3}(X,d)$, $\Bl_d = \{(\omega,f) \in Y_d \times M_d \mid \im(f)
\subset X_\omega \}$, and $Z_d^{(3)} = \{(\omega,x_1,x_2,x_3) \in Y_d
\times X^3 \mid x_i \in X_\omega \text{ for } 1 \leq i \leq 3 \}$.
These spaces define a generalization of the diagram (\ref{E:diagram}),
where the maps $\pi$, $\phi$, $\ev_i$, $e_i$, $p$, and $q$ are as
before defined using evaluation maps and projections.  If $T \subset
G$ is a maximal torus, then all of these maps are $T$-equivariant.
Part (d) of Proposition~\ref{P:cmp} implies that $\pi$ is birational,
and part (e) implies that $\phi$ is birational.  The following theorem
is proved exactly as Theorem~\ref{T:mainA}.

\begin{thm} \label{T:mainC}
  Let $d \leq d_{\max}$ be a small degree for the cominuscule variety
  $X$, and let $\alpha_1, \alpha_2, \alpha_3 \in K^T(X)$.  Then
  \[ I_d^T(\alpha_1, \alpha_2, \alpha_3) = 
     \euler{Y_d}^T(q_*p^*(\alpha_1) \cdot q_*p^*(\alpha_2) \cdot
     q_*p^*(\alpha_3)) \,.
  \]
\end{thm}

It would be possible to compute $K$-theoretic Gromov-Witten invariants
of large degrees if the following is true.

\begin{conj} \label{C:rational}
  Let $X$ be a cominuscule variety that is not a Grassmannian of type
  A and let $d > d_{\max}$.  If $x_1,x_2,x_3$ are general points in
  $X$, then the Gromov-Witten variety $GW_d(x_1,x_2,x_3)$ is
  rational.
\end{conj}

In fact, if $\ev : M_d \to X^3$ is the total evaluation map, then any
Gromov-Witten invariant can be written as
\[\begin{split}
   I_d^T(\alpha_1, \alpha_2, \alpha_3)
   &= 
   \euler{M_d}^T(\ev_1^*(\alpha_1) \cdot \ev_2^*(\alpha_2) \cdot
   \ev_3^*(\alpha_3)) \\
   &=
   \euler{X^3}^T(e_1^*(\alpha_1) \cdot e_2^*(\alpha_2) \cdot
   e_3^*(\alpha_3) \cdot \ev_*[\cO_{M_d}])
\end{split}\]
where $e_i : X^3 \to X$ is the $i$-th projection.  If
Conjecture~\ref{C:rational} is true, then Theorem~\ref{T:push} implies
that $\ev_*[\cO_{M_d}] = [\cO_{X^3}]$, and using Lemma~\ref{L:product}
we obtain:

\begin{cons} \label{C:kgwhigh}
  If $X$ and $d$ are as in Conjecture~\ref{C:rational} and $\alpha_1,
  \alpha_2, \alpha_3 \in K^T(X)$, then
  $I_d^T(\alpha_1,\alpha_2,\alpha_3) =
  \euler{X}^T(\alpha_1) \cdot \euler{X}^T(\alpha_2) \cdot
  \euler{X}^T(\alpha_3)$.
\end{cons}

In particular, it would follow that $I_d^T([\cO_{\Omega_1}],
[\cO_{\Omega_2}], [\cO_{\Omega_3}]) = 1$ for all $T$-stable Schubert
varieties $\Omega_1, \Omega_2, \Omega_3$ and $d > d_{\max}$.  As
mentioned in Remark~\ref{R:unirat}, we can prove that the 3-point
Gromov-Witten varieties for maximal orthogonal Grassmannians are
unirational, which suffices to establish Consequences~\ref{C:kgwhigh}
in this case.  The cohomological invariants of arbitrary degrees are
given by the following result.

\begin{thm}
  Let $X$ be a cominuscule variety and let $\beta_1, \beta_2, \beta_3 \in
  H^*_T(X)$.  Then
  \[ I_d^T(\beta_1, \beta_2, \beta_3) = \begin{cases}
    \int_{Y_d}^T q_*p^*(\beta_1) \cdot q_*p^*(\beta_2) \cdot
    q_*p^*(\beta_3) & \text{if $d \leq d_{\max}$} \\
    0 & \text{otherwise.}
    \end{cases}
\]
\end{thm}
\begin{proof}
  This is true if $X$ is a Grassmannians of type A by
  Theorem~\ref{T:mainA}, and if $d$ is a small degree by
  Theorem~\ref{T:mainC}.  If $X$ is not of type $A$ and $d >
  d_{\max}$, then a case by case check shows that $\dim(M_d) > 3
  \dim(X)$, which implies that $\ev_*[M_d] = 0$. 
\end{proof}

\begin{remark}
  Let $\Omega \subset X$ be a Schubert variety and set $\wt\Omega =
  q(p^{-1}(\Omega)) \subset Y_d$.  Then we have $q_* p^*([\Omega]) =
  [\wt\Omega]$ if (a translate of) $\Omega$ is contained in the dual
  Schubert variety of $X(d)$, and otherwise $q_* p^*([\Omega]) = 0$.
  This follows from \cite[(6)]{chaput.manivel.ea:quantum} together
  with Lemma~\ref{L:pushschub}.
\end{remark}



\providecommand{\bysame}{\leavevmode\hbox to3em{\hrulefill}\thinspace}
\providecommand{\MR}{\relax\ifhmode\unskip\space\fi MR }
\providecommand{\MRhref}[2]{%
  \href{http://www.ams.org/mathscinet-getitem?mr=#1}{#2}
}
\providecommand{\href}[2]{#2}


\end{document}